\documentclass[12pt,a4paper]{amsart}
 \usepackage{amsfonts,amssymb,enumerate,color,bm,hhline,makecell,array}
\usepackage{array,mathrsfs}
\usepackage[all,ps,cmtip]{xy}
\usepackage[normalem]{ulem}
\usepackage{bm,mathtools}
\usepackage{hyperref}

\usepackage{enumitem}
\setlist[itemize]{noitemsep,nolistsep}
\setlist[enumerate]{noitemsep,nolistsep}
\usepackage{colortbl}

\allowdisplaybreaks
\let\mathcal\mathscr

\def\Z{{\bf Z}}

\def\C{{\bf C}}

\def\P{{\bf P}}
\def\T{{\bf T}}

\def\phi{\varphi}

\def\bq{{\mathbf q}}

\def\cA{\mathcal{A}}

\def\cC{\mathcal{C}}
\def\cD{\mathcal{D}}

\def\cE{\mathcal{E}}

\def\cI{\mathcal{I}}

\def\cO{\mathcal{O}}

\def\cR{\mathcal{R}}

\def\cT{\mathcal{T}}
\def\cU{\mathcal{U}}
\def\cV{\mathcal{V}}

\def\lra{\longrightarrow}
\def\llra{\hbox to 10mm{\tofill}}
\def\lllra{\hbox to 15mm{\tofill}}

\def\llla{\hbox to 10mm{\leftarrowfill}}
\def\lllla{\hbox to 15mm{\leftarrowfill}}

\def\dra{\dashrightarrow}

\def\thra{\twoheadrightarrow}
\def\hra{\hookrightarrow}
\def\lhra{\ensuremath{\lhook\joinrel\relbar\joinrel\to}}

\def\isom{\simeq}
\def\eps{\varepsilon}

\def\tY{\widetilde{Y}}

\def\vide{\varnothing}
\def\emptyset{\varnothing}
\def\subset{\subseteq}
\def\supset{\supseteq}

\DeclareMathOperator{\isomto}{\stackrel{{}_{\scriptstyle\sim}}{\to}}

\DeclareMathOperator{\isomlra}{\stackrel{{}_{\scriptstyle\sim}}{\lra}}

\DeclareMathOperator{\Alb}{Alb}

\DeclareMathOperator{\EPW}{\mathbf{{M}}{}^{\mathrm{EPW}}}
\DeclareMathOperator{\bEPW}{\mathbf{\overline{M}}{}^{\mathrm{EPW}}}
\newcommand{\bcM}{{{\mathbf{M}}}^{\mathrm{GM}}}

\DeclareMathOperator{\Bl}{Bl}

\DeclareMathOperator{\Coker}{Coker}

\newcommand{\Db}{\mathrm{D^b}}
\newcommand{\Se}{\mathrm{S}}

\DeclareMathOperator{\Dis}{Disc}

\DeclareMathOperator{\Ext}{Ext}

\DeclareMathOperator{\Gr}{\mathsf{Gr}}
\DeclareMathOperator{\CGr}{\mathsf{CGr}}
\DeclareMathOperator{\LGr}{\mathsf{LGr}}

\DeclareMathOperator{\Hilb}{Hilb}
\DeclareMathOperator{\Hom}{Hom}
\DeclareMathOperator{\Id}{Id}
\def\Im{\mathop{\rm Im}\nolimits}

\DeclareMathOperator{\Ker}{Ker}

\DeclareMathOperator{\lin}{\underset{\mathrm lin}{\equiv}}

\DeclareMathOperator{\PGL}{PGL}
\DeclareMathOperator{\Pic}{Pic}

\DeclareMathOperator{\SO}{SO}

\DeclareMathOperator{\Prym}{Prym}

\DeclareMathOperator{\bSpec}{\mathbf{Spec}}

\DeclareMathOperator{\Sing}{Sing}

\DeclareMathOperator{\Sym}{Sym}

\def\llra{\hbox to 10mm{\tofill}}
\def\lllra{\hbox to 15mm{\tofill}}

\def\bw#1{\textstyle{\bigwedge\hskip-0.9mm^{#1}}}

\newtheorem{lemm}{Lemma}[section]
\newtheorem{theo}[lemm]{Theorem}
\newtheorem{coro}[lemm]{Corollary}
\newtheorem{prop}[lemm]{Proposition}

\theoremstyle{definition}
\newtheorem{defi}[lemm]{Definition}
\newtheorem{rema}[lemm]{Remark}

\newtheorem{conj}[lemm]{Conjecture}

\newtheorem{exam}[lemm]{Example}

\theoremstyle{remark}
\newtheorem*{remark*}{Remark}
\newtheorem*{note*}{Note}

\def\moins{\smallsetminus}
\def\setminus{\smallsetminus}

\newcommand{\gquot}{/\!\!/}

\begin{document}
\title{Gushel--Mukai varieties}

 \author[O.\ Debarre]{Olivier Debarre}
 
 \address{ Universit\'e Paris Cit\'e and Sorbonne Universit\'e, CNRS, IMJ-PRG, F-75013 Paris, France
} \email{{\tt  olivier.debarre@imj-prg.fr}}
 
\thanks{This survey was initially written for a series of lectures given at the Shanghai Center for Mathematical Sciences at  Fudan University in the fall of 2019.\ I would like to thank this institution and Zhi\ Jiang for the invitation and support,   Qixiao Ma for his many questions and remarks, Alexander Perry for his comments about competing conjectures for rationality of Gushel--Mukai fourfolds, Jean-Louis Colliot-Th\'el\`ene for his explanations about conic bundles, and, last but certainly not least, Alexander~Kuznetsov for correcting some of my mistakes and for the pleasure of working with him all these years.\\
\indent These notes were updated in 2025, while I  was supported by the European Research Council under the European Union's Horizon 2020 research and innovation programme (ERC-2020-SyG-854361-HyperK)}

\begin{abstract}
 Gushel--Mukai  varieties are smooth complex dimensionally transverse intersections of a cone over the Grassmannian $\Gr(2,5)$ with a linear space and a quadratic hypersurface.\ The aim of this survey is to discuss the geometry, moduli,  Hodge structures, and categorical aspects of these varieties.\ It is based on joint work with Alexander Kuznetsov and earlier work of Logachev, Iliev,  Manivel,   O'Grady, and others.
\end{abstract}

\makeatletter
\@namedef{subjclassname@2020}{%
  \textup{2020} Mathematics Subject Classification}
\makeatother


 \maketitle

{\renewcommand{\baselinestretch}{0.5}\normalsize
\tableofcontents}

\section{Introduction}

{\em Gushel--Mukai} (or GM for short) {\em varieties} are smooth complex dimensionally transverse intersections of a cone over the Grassmannian $\Gr(2,5)$ with a linear space and a quadratic hypersurface.\ They occur in each dimension 1 through 6 and they are Fano varieties (their anticanonical bundle is ample) in dimensions 3, 4, 5, and~6 (in this text, we mostly restrict ourselves to these dimensions).\ 

These varieties first appeared in the classification of complex Fano threefolds: Gushel showed that any smooth    Fano threefold of index 1, degree 10, and Picard number~$1$ is a GM variety.\ Mukai later extended Gushel's results and proved that all Fano varieties of coindex 3, degree 10, and Picard number~$1$, all Brill--Noether-general polarized K3 surfaces of degree 10, and all Clifford-general curves of genus 6 are  GM varieties. 

One of the reasons why their geometry is so rich is that, to any GM variety is canonically attached a sextic hypersurface in ${\bf P}^5$, called an {\em Eisenbud--Popescu--Walter} (or EPW for short) {\em sextic} and a canonical double cover thereof, a
hyper-K\"ahler fourfold called a {\em double EPW sextic.}\ In some sense, the pair consisting of a GM variety and its associated double EPW sextic behaves very much (but with some complications and     differences) like a cubic hypersurface in ${\bf P}^5$ and its (hyper-K\"ahler) variety of lines.
Not many examples of these pairs---a Fano variety and a hyper-K\"ahler variety---are known (another example is given by the hyperplane sections of the Grassmannian $\Gr(3,10)$ and their associated Debarre--Voisin hyper-K\"ahler fourfold), so it is worth looking in depth into one of these.

The first main difference with cubic fourfolds is that there is a positive-dimensional family of GM varieties attached to the same EPW sextic; this family can be precisely described.\ Another distinguishing feature is that each EPW sextic has a dual EPW sextic (its projective dual).\ GM varieties associated with isomorphic or dual EPW sextics are all birationally isomorphic.

A   feature shared with cubic fourfolds is that the middle Hodge structure of a GM variety of {\em even} dimension is isomorphic (up to a Tate twist) to the second cohomology of its associated double EPW sextic.\ Together with the Verbitsky--Torelli theorem for hyper-K\"ahler fourfolds, this leads to a complete description of the period map for   GM varieties of dimensions 4 or 6.

In  {\em odd} dimensions 3 or 5, a GM variety has a 10 dimensional intermediate Jacobian and we show that it is isomorphic to the Albanese variety of a surface of general type canonically attached to the EPW sextic.\ So again, the Hodge structure of the GM variety is determined by  the associated EPW sextic.

Another aspect of GM varieties that makes them very close to cubic fourfolds is the  rationality problem.\ Whereas general (and conjecturally all) GM varieties of dimensions at most~3 are not rational, and GM varieties of dimensions at least 5 are all rational, the situation for GM fourfolds is still mysterious: some  rational examples are known but not a single irrational example is known, although one expects, as for cubic fourfolds, that a very general GM fourfold should be irrational.

Finally,  the derived categories of GM varieties were studied by Kuznetsov and Perry with this rationality problem in mind.\ They show that  each of these   categories contains a special subcategory, which is a K3 or Enriques category according to whether the dimension of the variety is even or odd.

Most of  my  contributions are joint work with Alexander Kuznetsov.

\section{Definition of Gushel--Mukai varieties}

We always work over $\C$ and $U_k$, $V_k$, $W_k$ all denote complex vector spaces of dimension $k$.

Let $X$ be a (smooth complex) Fano variety of dimension $n$ and Picard number $1$, so that $\Pic(X)=\Z H$, with $H$ ample, and  $-K_X\lin rH$ for some  positive integer $r$ called the {\em index} of~$X$.\ The integer $d\coloneqq H^n$ is the {\em degree.}\
For example, a smooth hypersurface $X\subset \P^{n+1}$ of degree $d\le n+1$ is a Fano variety of index $r=n+2-d$ and degree $d$.

 The index satisfies $r\le n+1$ and moreover,   
\begin{itemize}
\item if $r= n+1$, then $X\isom \P^n$ (Kobayashi--Ochiai, 1973; see \cite[Corollary to~Theorem~1.1]{ko});
\item if $r= n$, then $X$ is a  quadric in $ \P^{n+1}$ (Kobayashi--Ochiai, 1973; see \cite[Corollary to~Theorem~2.1]{ko});
\item if $r= n-1$ (del Pezzo varieties), there is a classification by Fujita and Iskovskikh (1977--1988; see \cite[Chapter~3]{ip});
\item if $r= n-2$, there is a  classification by Mukai (1989--1995; see \cite{muk1, muk2, muk3}),\footnote{Mukai's argument is in fact incomplete (see footnote~\ref{f1}).\ For a complete proof of a more general result, see \cite{bkm}.} modulo a result  proved later by Shokurov in dimension 3, and Mella in all dimensions~(\cite{mel}).
\end{itemize}

Mukai used the {\em vector bundle method} (initiated by Gushel in \cite{gush}) to prove that in the last case, $d\in \{2,4,6,8,10,12,14,16,18,22\}$.\ 

In this survey, we restrict ourselves  to the case $d=10$.\ We call the corresponding (smooth Fano) varieties {\em Gushel--Mukai} (or GM for short) varieties.

\begin{theo}[Mukai]\label{theo11}
Any smooth Fano $n$fold $X_n$ with Picard number 1, index $n-2$, and degree~$10$ has dimension $n\in\{3,4,5,6\}$ and can be obtained as follows:
\begin{itemize}
\item either $n\in\{3,4,5\}$ and $X_n=\Gr(2,V_5)\cap \P(W_{n+5})\cap Q\subset \P(\bw2V_5)$, where $Q$ is a quadric and $W_{n+5}\subset \bw2V_5$ is a vector subspace of dimension $n+5$ ({\em ordinary} case);
\item or $X_n\to \Gr(2,V_5)\cap \P(W_{n+4})$ is a double cover branched along an $X_{n-1}$ ({\em special} case).
\end{itemize}
\end{theo}

In fact, Mukai had to make the extra assumption that the linear system $|H|$ contains a smooth element, a fact that was proved only later on by Mella (who actually showed that $H$ is very ample).\ 

In both cases, the morphism $\gamma\colon  X_n\to \Gr(2,V_5)\subset \P(\bw2V_5)$ (called the {\em Gushel morphism}) is defined by   a linear subsystem of $|H|$. 

In dimension $n=1$, these constructions provide   genus 6 curves which are general in the sense that they are neither   hyperelliptic, nor trigonal, nor  plane quintics (this is also termed {\em Clifford general} because the Clifford index has the maximal value $2$); in dimension $n=2$, these constructions provide   K3 surfaces of degree 10 which are general in the sense that their general hyperplane sections are Clifford general genus 6 curves (they are also called {\em Brill--Noether general} K3 surfaces; see \cite[Definition~3.8]{muk3}).

\begin{proof}[Sketch of proof]
Set $X\coloneqq X_n$.\ In order to construct the Gushel morphism $\gamma\colon X\to \Gr(2,V_5)$ with $\gamma^*\cO_{\Gr(2,V_5)}(1)=H$, one needs  to  construct a rank 2 globally generated vector bundle $\cE$ on~$X$ with 5 independent sections and $c_1(\cE)=H$; if $\cU$ is the tautological rank 2 vector bundle on $ \Gr(2,V_5)$, one then has $\cE=\gamma^*\cU^\vee$.\ By Mella's result, the line bundle $H$ is very ample hence defines an embedding $\phi_H\colon X\hra \P(W_{n+5})$, where $W_{n+5}\coloneqq  H^0(X,H)^\vee$.

We give in \cite{dk1} a construction of $\cE$ based on the excess conormal  sheaf: we prove that $X\subset  \P(W_{n+5})$ is the base locus of the space of quadrics $V_6\coloneqq  H^0(\P(W_{n+5}),\cI_X(2))$ and we define
$\cE$ as the dual of the (rank 2) kernel of the canonical surjection
$$V_6 \otimes \cO_X \thra (\cI_X/\cI_X^2 )(2)= N^\vee_{X/\P(W_{n+5})}(2).$$

 Mukai had a different construction   that went roughly as follows (\cite{muk1, muk3}).\ The intersection of $n-2$ general elements of $|H|$ is a smooth K3 surface~$S$ and the restriction $\cE_S$ of $\cE$ to~$S$ must satisfy  $c_1(\cE_S)=H_S$ and $c_2(\cE_S)=4$.

 So we start out by constructing such a vector bundle on the genus 6 K3 surface $S\subset \P^6$ using  Serre's construction.\ A general hyperplane section  of $S$ is a canonical curve $C$ of genus 6, hence has a $g^1_4$.\ Let $Z\subset C$ be an element of this pencil, so that 
  \begin{eqnarray*}
h^0(S, \cI_Z(H_S))&=&1+h^0(C,H_S-Z)\\
&=&1+h^1(C,g^1_4)\\
&=&1+h^0(C,g^1_4)-(4+1-6)=4.
\end{eqnarray*}
By Serre's construction, there is a vector bundle $ \cE_S$ on $S$ that fits into an exact sequence
 $$0\to \cO_S\to \cE_S\to \cI_Z(H_S)\to 0.$$
 
 Assume $n=3$.\ One proves that $\cE_S$   extends   to a locally free sheaf $\cE$ on~$X$.\footnote{\label{f1}The argument given by Mukai for this step is incorrect, since it relies on an extension theorem   of Fujita that does not apply to surfaces.\ It was only recently fixed in \cite{bkm}.} Then one checks that $V_5\coloneqq  H^0(X,\cE)^\vee$ has dimension 5 and that $\cE$ is generated by global sections hence defines
 a morphism $\gamma\colon X\to \Gr(2,V_5)$.
 
  We have a linear map 
   $$\eta\colon \bw2 H^0(X,\cE) \to H^0(X,\bw2\cE)\isom H^0(X,H) 
 $$
and  a commutative diagram 
\begin{equation*}\label{diag}
\xymatrix
@C=50pt@M=6pt
{
X\ar[r]^\psi\ar@{_(->}[d]_{\phi_H}&\Gr(2,V_5)\ar@{_(->}[d]\\
\P (H^0(X,H)^\vee)\ar@{-->}[r]^{\eta^\vee}&\P(\bw2V_5).}
\end{equation*}

{\em If $\eta$ is surjective,} the  map $\eta^\vee$ in the diagram    is a morphism, which is moreover injective.\ The image of $\eta^\vee$ is a    $\P(W_8)\subset \P(\bw2V_5)$.\ The intersection $M_X\coloneqq \Gr(2,V_5)\cap \P(W_8)$ has dimension~4 (because all hyperplane sections of $ \Gr(2,V_5)$ are irreducible) and degree 5, and $\phi_H(X) \subset M_X$ is a smooth hypersurface.\ It follows that $M_X$ must be smooth (otherwise, one easily sees that its singular locus would have positive dimension hence would meet $\phi_H(X)$, which is impossible).\ Thus, $\phi_H(X)$ is a Cartier divisor, of  degree 10, hence is the intersection of $M_X$   with a quadric.

{\em If the corank of $\eta$ is 1,}  consider the cone
$\CGr\subset \P(\C\oplus \bw2V_5)$, with vertex $v=\P(\C)$, over  $\Gr(2,V_5)$.
 In the above diagram, we may view $H^0(X,H)^\vee$ as embedded in $\C\oplus \bw2V_5$ by mapping $(\Im (\eta))^\bot$ to $\C$, and $\eta^\vee$ as the projection from $v$, with image a $\P(W_7)\subset \P(\bw2V_5)$.\ Again, the intersection $M_X\coloneqq \Gr(2,V_5)\cap \P(W_7)=\Im(\psi)$ is a smooth threefold and $\phi_H(X)$ is the intersection of   the (smooth locus of the) degree 5 cone over $M_X$  with a quadric.

 {\em If the corank of $\eta$ is $ \ge 2$,} one shows that $\phi_H(X)$ must be singular, which is impossible.\ So this proves the theorem when $n=3$.
 
For $n\ge 4$, one extends $\cE$ to successive hyperplane sections of $X$ all the way up to $X$ and proceed similarly (see \cite[Proposition 5.2.7]{ip}).
 \end{proof}

Why are people interested in GM varieties? As the rest of this survey will, I hope, show,  for several reasons:
\begin{itemize}
\item they have interesting geometries (Section~\ref{se2});
\item they have intriguing rationality properties in dimension 4 similar to that of cubic fourfolds (most $X_3$ are irrational, all $X_5$ and $X_6$ are rational) (Section~\ref{se3});
\item they have interesting period maps (Section~\ref{se4});
\item they have interesting derived categories (Section~\ref{se5}).
\end{itemize}

\section{Gushel--Mukai varieties,   EPW sextics, and moduli}\label{se2}

 In this section, we explain an intricate (but   elementary) connection (discovered by O'Grady and Iliev--Manivel) between Gushel--Mukai varieties and Lagrangian subspaces in a 20 dimensional symplectic complex vector space.
 
 Let 
 $$ X=\Gr(2,V_5)\cap \P(W_{n+5})\cap Q\subset \P(\bw2V_5)
 $$ 
 be an ordinary\footnote{We make this assumption for the sake of simplicity, but our  constructions also work for special GM varieties, with some modifications (see Remark~\ref{remspecial}).\  The theory also works    in dimensions~$n\in\{1,2\}$, again with some modifications.\ We refer to \cite{dk1} for details.} GM $n$fold, with $n\in\{3,4,5\}$.\ We set
  $$M\coloneqq  \Gr(2,V_5)\cap \P(W_{n+5}),$$ 
  a variety of degree 5 and dimension $n+1$, so that $X=M\cap Q$.
  
  \begin{lemm}\label{l1}
The intersection $M$ is smooth and each nonzero element of the orthogonal complement $W_{n+5}^\perp \subset \bw2{V_5}^\vee$, viewed as a skew-symmetric form on $V_5$, has maximal rank 4.
\end{lemm}

\begin{proof}
Let $W\coloneqq W_{n+5}$.\ One has   
(see \cite[Corollary~1.6]{pv}) 
\begin{equation}\label{esing}
\Sing(M)=\P(W)\ \cap\hskip-3mm \bigcup_{\omega\in W^\bot\setminus\{0\}}\hskip-4mm \Gr(2,\Ker(\omega )).
\end{equation}
Since $X$ is smooth, $\Sing(M)$ must be  finite.

 If $\dim (W) = 10$, then $M= \Gr(2,V_5)$ is smooth. 
 
If $\dim (W) = 9$, either a generator $\omega$ of~$W^\bot$ has rank 2 and $\Sing(M)$ is a  plane, which is impossible, or else $M$ is smooth. 

If $\dim (W) = 8$, either
  some  $\omega\in W^\bot$ has rank 2, in which case $\Gr(2,\Ker(\omega ))$ is a  plane contained in the hyperplane~$\omega^\bot$, whose intersection 
 with $\P(W)$ therefore contains a line along which $M$ is singular,  which is impossible, or else $M$ is smooth.
\end{proof}

\subsection{GM data sets}\label{s21}

Note that $M$ is cut out in $\P(W)$ by the traces of the Pl\"ucker quadrics
$$\bq(v)(w,w)\coloneqq v\wedge w\wedge w,$$
for $v\in V_5$ and $w\in W\subset \bw2V_5$ (we choose an isomorphism $\bw5V_5\isom\C$).\ More precisely, there is a (not quite canonical) identification
$$V_5\isom H^0(\P(W),\cI_M(2)).$$
One then has 
$$V_6:\isom H^0(\P(W),\cI_X(2))\isom V_5\oplus \C Q.$$

  We formalize this by defining a {\em GM data set} $(W_{n+5},V_6,V_5,\bq)$ (of dimension $n\le5$) by
  \begin{itemize}
\item $V_6$ is a 6 dimensional vector space;
\item $V_5\subset V_6$ is hyperplane;
\item $W_{n+5}\subset \bw2V_5$ is a vector subspace of dimension $n+5$;
\item $\bq\colon V_6\to \Sym^2 W_{n+5}^\vee$ is a linear map such that
$$\forall v\in V_5\ \ \forall w\in W_{n+5}\qquad \bq(v)(w,w)=v\wedge w\wedge w.$$
\end{itemize}
Given such a data set, we define
$$X\coloneqq \bigcap_{v\in V_6} Q(v)\subset \P(W_{n+5}),$$
where $ Q(v)\subset \P(W_{n+5})$ is the quadric hypersurface defined by $\bq(v)$.\ When $X$ is smooth of the expected dimension $n$, it is an ordinary GM variety.

\subsection{Lagrangian data sets}\label{s22}
 
We now set up a  more complicated correspondance.\ We   endow~$\bw3V_6$ with the (conformal) symplectic form given by wedge product.\ A {\em Lagrangian data set} is a triple $(V_6,V_5,A)$, where $V_5\subset V_6$ is a hyperplane and $A\subset \bw3V_6$ is a (10 dimensional) Lagrangian subspace.

\begin{theo}\label{th22}
There is a bijection (explained in the proof) between GM data sets $(W,V_6,V_5,\bq)$  and Lagrangian data sets $(V_6,V_5,A)$.

One has the following   relations
\begin{eqnarray*}
&\dim(W)= 10-\dim (A\cap \bw3V_5),\\
&\forall v\in V_6\moins V_5\qquad \Ker(\bq(v))\isom A\cap (v\wedge \bw2V_6)=A\cap (v\wedge \bw2V_5)\subset W
.
\end{eqnarray*}
\end{theo}

\begin{proof} This construction is explained in great details and more generality in the proof of \cite[Theorem~3.6]{dk1}.\ We present a simplified version.

For one direction, let us start with a Lagrangian data set $(V_6,V_5,A)$.\ Let $\lambda\in V_6^\vee$ be a linear form with kernel $V_5$.\ By the Leibniz rule, it induces $\lambda_p\colon \bw{p}V_6\to \bw{p-1}V_5$, with kernel~$\bw{p}V_5$.\ Define
$\bq\colon V_6\otimes \Sym^2\!A\to \C$ by
$$\forall v\in V_6\ \ \forall \xi,\xi'\in A\qquad \bq(v\otimes (\xi,\xi'))\coloneqq -\lambda_4 (v\wedge \xi)\wedge \lambda_3(\xi')\in \bw5V_5.$$
The fact that $A$ is Lagrangian implies $\xi\wedge\xi'=0$, hence, by the Leibniz rule,
$$0=\lambda_6 (\xi\wedge\xi')=\lambda_3 (\xi)\wedge\xi'-\xi\wedge\lambda_3 (\xi').$$
Therefore, we have, by the Leibniz rule again,
$$\bq(v\otimes (\xi,\xi'))\coloneqq -\lambda_1 (v)\wedge \xi\wedge \lambda_3(\xi')+v\wedge\lambda_3 ( \xi)\wedge \lambda_3(\xi'),
$$
which is symmetric in $\xi$ and $\xi'$.

Moreover, if $\xi'\in A\cap \bw3V_5$, one has $\lambda_3(\xi')=0$, hence $\bq(v\otimes (\xi,\xi'))=0$, so that $\bq$ descends to
$$\bq\colon V_6\otimes \Sym^2 W\to \C,$$
where $W\coloneqq A/(A\cap \bw3V_5)$.\ 

Finally, if $v\in V_5$, one has $\lambda(v)=0$, hence $\bq(v\otimes (\xi,\xi'))=v\wedge\lambda_3 ( \xi)\wedge \lambda_3(\xi')$ by the calculation above.\ The canonical factorization of the composition
\begin{equation*}\label{eqw}
A\subset \bw3V_6\xrightarrow{\ \lambda_3\ }\bw2V_5
\end{equation*}
induces an injection $W\hra \bw2V_5$ and all the conditions are met for  $(W,V_6,V_5,\bq)$ to be a GM data set.
\medskip

For the other direction, we start from a GM data set $(W,V_6,V_5,\bq)$.\ Choose $v_0\in V_6\moins V_5$.\ We define $A$ by the short exact sequence
\begin{equation}\label{eqA}
\begin{array}{ccccccccc}
 0 &  \to & A&\to &\bw3V_5\oplus W&\lra &W^\vee  \\
  &&&& (\xi,w)&\longmapsto&\bigl(w'\mapsto \xi\wedge w'+\bq(v_0)(w,w')
  \bigr)   ,
\end{array}
\end{equation}
where the middle term is seen inside $\bw3V_5\oplus( v_0\wedge \bw2V_5)=\bw3V_6$.

Let us check that $A$ is Lagrangian.\ If $(\xi,w),(\xi',w')\in A$, we have
\begin{eqnarray*}
(\xi,w)\wedge(\xi',w')&=&(\xi+v_0\wedge w)\wedge(\xi'+v_0\wedge w')\\
&=&\xi\wedge v_0\wedge w'+v_0\wedge w\wedge\xi'\\
&=&v_0\wedge(-\xi\wedge w'+\xi'\wedge w).
\end{eqnarray*}
But this vanishes because   $(\xi,w)$ goes to 0 in $W^\vee$, hence $ \xi\wedge w'+\bq(v_0)(w,w')=0$, and similarly $ \xi\wedge w+\bq(v_0)(w',w)=0$.\ So the fact that $A$ is isotropic follows from the symmetry of $\bq(v_0)$.\ Since $A$ has dimension $\ge 10$ by construction, it is Lagrangian (and the rightmost map in \eqref{eqA} is surjective). 

One checks that $A$ is independent of the choice of $v_0$ and that the two constructions are inverse of each other.

The formula for $\dim(W)$ was given in the course of the proof.\ By construction of $A$, the kernel of $\bq(v_0)$ is the set of pairs~$(0,w)$ that belong to $A$, which is isomorphic to $A\cap (v_0\wedge \bw2V_5)$.\ Since $A$ is independent of the choice of $v_0$, this remains valid for all $v_0\in V_6\moins V_5$.
\end{proof}

Putting everything together, we obtain a bijection between the set of isomorphism classes of (possibly singular) ordinary GM varieties and the set of isomorphism classes of  certain Lagrangian data sets $(V_6,V_5,A)$.\ We need to worry about   smoothness of the associated GM variety.\ It turns out, and this is what makes this construction so efficient, that there is a very simple criterion for smoothness that only involves the Lagrangian~$A$ and not the hyperplane $V_5$.

\begin{theo}\label{th23}
Let $n\in\{3,4,5\}$.\ 
The bijections constructed above combine to give a bijection between the set of isomorphism classes of   ordinary GM varieties of dimension $n$ and the set   of isomorphism classes of Lagrangian data sets $(V_6,V_5,A)$ such that 
$\dim (A\cap \bw3V_5)=5-n$ and~$A$ {\em contains no decomposable vectors.}
\end{theo}

The condition on $A$ means $\P(A)\cap \Gr(3,V_6)=\vide$ inside $\P(\bw3V_6)$.

\begin{proof}
Let $X$ be an ordinary GM variety of dimension $n$.\ Since $X$ is smooth, by Lemma~\ref{l1}, 
  $M= \Gr(2,V_5)\cap \P(W)$ is also smooth and 
$W^\bot$ contains no rank 2 forms, that is, $\P(W^\bot)\cap \Gr(2,V_5^\vee)=\vide$.

If $\xi\in \P(A)\cap \Gr(3,V_5)$, the exact sequence \eqref{eqA} implies that the linear form $w'\mapsto \xi \wedge w'$ is in $W^\perp$; but it has rank 2, and this is absurd.

If $\xi\in A$ is decomposable, it is therefore not in $\Gr(3,V_5)$, so we can write $\xi=v_0\wedge v_1\wedge v_2$ for some $v_0\in V_6\moins V_5$ and $v_1,v_2\in V_5$.\ From the exact sequence \eqref{eqA} defining $A$, we obtain that $v_1\wedge v_2$ is in $W$ and that $(0, v_1\wedge v_2)$ goes to $0$ in $W^\vee$.\ The bivector  $v_1\wedge v_2$ is therefore in the kernel of the quadratic form $\bq(v_0)$.\ In other words, the quadric~$Q(v_0)$ is singular at the point $[v_1\wedge v_2]$ of $\P(W)$.\ But this point is on $M$, hence it is singular on $X=M\cap Q(v_0)$, which is absurd.

All this shows that $A$ contains no decomposable vectors.

\medskip

Assume for the converse that $A$ contains no decomposable vectors.\ What we just showed implies that $M$ is smooth of dimension $n+1$ (because $A\cap \Gr(3,V_5)=\vide$).\ If $X$ is singular at a point $[w]$, where $w\in \Gr(2,V_5) $, this therefore means that $\T_{Q(v),[w]}\supset \T_{M,[w]}$ for all  $v\in V_6\moins V_5$.\ Since $M$ is the intersection of Pl\"ucker quadrics, there exists $v'\in V_5$ such that $\T_{Q(v),[w]}=\T_{Q(v'),[w]}$.\ Replacing~$v$ by some  $v+tv'$, we see
  that there exists $v\in V_6\moins V_5$ for which the quadric~$Q(v)$ is singular at~$[w]$.\ The second relation in Theorem~\ref{th22} implies that $v\wedge w$ is in~$A$, which is absurd since it is decomposable.
\end{proof}

\begin{rema}[Special GM varieties]\label{remspecial}
Recall from Theorem~\ref{theo11}  that a special GM variety $X$ of dimension $n$ is canonically a double cover of a linear section of the Grassmannian $\Gr(2,V_5)$ branched along an ordinary GM variety $X'$ of dimension $n-1$.\ One defines the Lagrangian data set   associated with $X$ as the Lagrangian data set   associated with $X'$ by the procedure described above (see \cite[Section~3]{dk1}).
\end{rema}
  
\subsection{EPW sextics}  Eisenbud--Popescu--Walter (EPW) sextics are special hypersurfaces of degree~6 in $\P(V_6)$   constructed from  Lagrangian subspaces $A \subset \bw3V_6$.\ More generally, given such a Lagrangian $A$, we consider, for any integer $\ell$,  the  locus
\begin{equation}\label{yabot}
Y_A^{\ge \ell}\coloneqq \bigl\{[v]\in\P(V_6) \mid \dim\bigl(A\cap (v \wedge\bw{2}{V_6} )\bigr)\ge \ell\bigr\}
\end{equation}
and endow it with a   scheme structure as in \cite[Section 2]{og1} (see also Section~\ref{se34}).\ We also set $ Y_A^{ \ell}\coloneqq Y_A^{\ge \ell}\smallsetminus Y_A^{\ge \ell+1}$.\ We have inclusions 
\begin{equation*}
\P(V_6) = Y_A^{\ge 0} \supset Y_A^{\ge 1} \supset Y_A^{\ge 2} \supset \cdots
\end{equation*}
and when the scheme $Y_A \coloneqq  Y_A^{\ge 1}$ is not the whole space $\P(V_6)$, it is a sextic hypersurface \mbox{(\cite[(1.8)]{og1})} called an {\em EPW sextic}.\ When $A$ contains no decomposable vectors, the geometry of these loci was described by O'Grady.

\begin{theo}[O'Grady]\label{th24}
Let $A \subset \bw{3}{V_6}$
be a Lagrangian subspace with no decomposable vectors. 
\begin{itemize}
\item[\rm (a)] $Y_A $ is an integral normal sextic hypersurface in $\P(V_6)$;
\item[\rm (b)] $Y_A^{\ge 2} = \Sing(Y_A )$ is an integral normal  Cohen--Macaulay surface of degree $40$;
\item[\rm (c)] $Y_A^{\ge 3} = \Sing(Y_A^{\ge 2})$ is  finite and smooth,  and is empty for  $A$ general;
\item[\rm (d)] $Y_A^{\ge 4} $  is empty.
\end{itemize}
\end{theo}

Given an ordinary  GM variety $X$, with associated Lagrangian $A$, Iliev and Manivel described in \cite{im} a direct way to construct
 its associated EPW sextic $Y_A$.\ It goes as follows.\ The space 
$V_6 $ of quadrics  in $\P(W)$ containing~$X$ is $6$ dimensional.\ Define the  discriminant locus $\widetilde\Dis(X)$ as  
the subscheme of~$\P(V_6)$ of {\em singular} quadrics containing~$X$.\ It is  
a hypersurface of degree $\dim(W) = n + 5$ in which the multiplicity  
of the hyperplane $\P(V_5)$ of  Pl\"ucker quadrics is  {at least} the corank of a general such quadric, 
which is at least $\dim(W)-6=n-1$.\ We set
\begin{equation}
\Dis(X) \coloneqq   \widetilde\Dis(X) -(n-1) \P(V_5).  
\end{equation}
This is an effective divisor of degree $6$   in $\P(V_6)$.

\begin{prop} 
Let $X$ be an ordinary GM variety of dimension $n\in\{3,4,5\}$ with associated Lagrangian $A$.\
The subschemes $Y_{A}$ and $\Dis(X)$ of $\P(V_6)$ are  equal.
\end{prop}

\begin{proof}
For all $v\in V_6\moins V_5$, we have $\Ker (\bq(v)) = A \cap (v \wedge \bw2{V_5})$ by Theorem~\ref{th22} hence, 
 $Y_A$ and $\Dis(X)$ coincide as sets on the complement of $\P(V_5)$.\ Since they are both sextics and $Y_A$ is integral (Theorem~\ref{th24}(a)) and not contained (as a set) in $\P(V_5)$, they are equal.
\end{proof}

\subsection{Dual EPW sextics}\label{s24} If $A\subset \bw3V_6$ is a Lagrangian subspace, so is its orthogonal $A^\bot 
 \subset \bw3V_6^\vee$ (which we also call its {\em dual}) and $A$ contains decomposable vectors if and only if $A^\bot$ does.\ The EPW loci~\eqref{yabot} for~$A^\bot$ can be described  {in terms of $A$} as
\begin{equation}\label{dualYA}
Y_{A^\perp}^{\ge \ell} = \big\{ [V_5] \in \Gr(5,V_6)=\P(V_6^\vee) \mid \dim (A \cap \bw{3}{V_5}) \ge \ell \big\}
\end{equation}
and the sextic hypersurface $Y_{A^\perp}\subset \P(V_6^\vee)$ is the projective dual of the hypersurface 
$Y_{A}\subset \P(V_6)$.

In particular, for each $n\in\{3,4,5\}$, we can rewrite the bijection of Theorem~\ref{th23} as a bijection between   the set of isomorphism classes of   ordinary GM varieties of dimension~$n$ and the set of isomorphism classes of Lagrangian data sets $(V_6,V_5,A)$ such that  $A$  contains no decomposable vectors and $[V_5]\in Y_{A^\perp}^{5-n} 
$.

\subsection{Moduli}\label{se25}
A GIT moduli space for EPW sextics was constructed by O'Grady in \cite{og5}: consider the natural action of the group $\PGL(V_6)$ on the Lagrangian Grassmannian~$\LGr(\bw3V_6)$
and its natural linearization in the line bundle $\cO(2)$ (the line bundle $\cO(1)$ has no such linearizations).\ The GIT quotient
\begin{equation*}
\bEPW  \coloneqq \LGr(\bw3V_6)\gquot\PGL(V_6)
\end{equation*}
is a  projective, irreducible, 20 dimensional, coarse moduli space for   EPW sextics.\
The hypersurface 
\begin{equation*}
\Sigma \coloneqq  \{ [A] \in \LGr(\bw3V_6)\mid \text{$A$ has decomposable vectors} \}
\end{equation*}
is $\PGL(V_6)$-invariant.\ Consider its complement: it
is affine, consists of stable points (\cite[Corollary~2.5.1]{og5}), and its image $\EPW $ in $ \bEPW$ is  affine and  is a coarse moduli space for EPW sextics~$Y_A$ such that $A$ has  no decomposable vectors.

The duality map $A\mapsto A^\bot$ described in Section~\ref{s24} induces a nontrivial involution $r$ on~$\bEPW$ that satisfies $r(\Sigma)=\Sigma$, hence also an involution on $\EPW$.

We do not know how to construct directly a moduli space for GM varieties.\ Instead, we extend the constructions of Section~\ref{s22} to families of GM varieties (this is not an easy task, not least because we need to include special GM varieties).\ Here is the main result of \cite{dk3}.

\begin{theo}\label{th26}
Let    $n\in\{3,4,5,6\}$.\
There is a quasiprojective, irreducible, coarse moduli space $\bcM_n$ for  GM varieties of dimension $n$.\ Its dimension is $25 - (5-n)(6-n)/2$, there is a surjective morphism
$$\pi_n\colon \bcM_n\lra \EPW,$$
and the fiber $\pi_n^{-1}([A])$ of the class of a Lagrangian $A$ with no decomposable vectors is isomorphic to $Y_{A^\perp}^{5-n} \cup Y_{A^\perp}^{6-n}$.\footnote{Modulo the finite group of linear automorphisms of $V_6$ that preserve $A^\bot$, which is trivial for $A$ very general (\cite[Proposition~B.9]{dk1}).}
\end{theo}

The piece $Y_{A^\perp}^{5-n}$ of the fiber corresponds to ordinary GM $n$folds and the piece~$Y_{A^\perp}^{6-n}$     to special   GM $n$folds.\ Also, the schemes $\bcM_6$ and $\bcM_5$ are affine\footnote{It should be remarked that $\bcM_6$ can  be directly constructed  as a standard GIT quotient of an affine space (\cite[Proposition~5.16]{dk3}).} and there is an open embedding $\bcM_6\hra \bcM_5$ whose image is the moduli space for ordinary GM fivefolds.

\section{Birationalities and rationality}\label{se3}
In this section, we will show that all GM varieties of the same dimension $n\in\{3,4,5,6\}$  with the same associated Lagrangian (that is, in the same fiber of the map $\pi_n$ defined in Theorem~\ref{th26}), or with dual associated Lagrangians, are birationally isomorphic.\ Since  all GM varieties of dimensions 5 or 6 are  rational (see Proposition~\ref{p33}), this is meaningful only \mbox{when $n\in\{3,4\}$.} 

Let $X$ be a GM variety of dimension $n\in\{3,4,5\}$, which we assume to be ordinary for simplicity, and let $\cU_X$ be the restriction to $X$   of the tautological rank 2 vector bundle on $\Gr(2,V_5)$.\ There are canonical maps
\begin{equation}\label{rho}
\vcenter
{\xymatrix@R=5mm{
  & \P(\cU_X)  \ar[dr]^{\rho} \ar[dl] & \\
  X&&\P(V_5).}
}
\end{equation}
Let 
\begin{equation*} 
  \Sigma_1(X) \subset \P(V_5)
\end{equation*}
be the union of  the kernels of all nonzero elements of  $W^\perp$, seen as rank 4 skew-symmetric forms on $V_5$ (see Lemma~\ref{l1}).\ It is
\begin{itemize}
\item empty for $n=5$ ($W^\bot=0$);
\item a point  for $n=4$;
\item a smooth conic for $n=3$.
\end{itemize}

\begin{prop}\label{prop31}
Let $X$ be an ordinary GM variety of dimension~$n\in\{3,4,5\}$.\ The map~$\rho$ is surjective and its fiber over a point $[v] \in\P(V_5)$ is
\begin{itemize}
\item if  $[v] \notin\Sigma_1(X)$,   a quadric in a $\P^{n-2}$, whose  corank is $k$ if and only if $[v] \in Y^k_A   $;
\item if $[v] \in\Sigma_1(X)$, a quadric in a  $\P^{n-1}$, whose  corank is $k$ if and only if  $[v] \in Y^{k+1}_A  $.\ In particular, $\Sigma_1(X)\subset Y_A$.
\end{itemize}
%
\end{prop}

\begin{proof}
Choose  $v_0 \in V_6 \setminus V_5$ and let $v \in V_5\moins\{0\}$.\
The fiber $\rho^{-1}([v])$ is the set of  $V_2\subset V_5$ containing~$v$ which correspond to points of $X$.\ This is the intersection $Q'(v)$ of the nonPl\"ucker quadric $Q(v_0) \subset \P(W)$ defining $X$  
with the subspace $\P\bigl( (v \wedge V_5) \cap W\bigr) \subset \P(W)$.

If $[v]\notin \Sigma_1(X)$, it is not in the kernel of any 2 form defining $W$, hence $\P\bigl( (v \wedge V_5) \cap W\bigr)$ has codimension $5-n$ in $\P (v \wedge V_5)\isom \P^3$: it is a linear space of dimension $n-2$.

If $[v]\in \Sigma_1(X)$, it is  in the kernel of a unique 2 form defining $W$, hence $\P\bigl( (v \wedge V_5) \cap W\bigr)$ has codimension $4-n$ in $\P (v \wedge V_5)\isom \P^3$: it is a linear space of dimension $n-1$.

The statements about the rank of $Q'(v)$ are not so obvious and we skip their proofs (see~\cite[Proposition~4.5]{dk1}).
 \end{proof}

\subsection{Birationalities} 
Our main result is the following (\cite[Corollary~4.16 and Theorem~4.20]{dk1}).

\begin{theo}\label{th32}
GM varieties of the same dimension $n\ge 3$ whose associated Lagrangians are either the   same or dual one to another are birationally isomorphic.
\end{theo}

\begin{proof}
Because of Proposition~\ref{p33}, the only interesting cases are when $n\in\{3,4\}$.\ We will only sketch the proof in the case when  $X$ and $X'$ are   ordinary GM threefolds associated with the same Lagrangian subspace $A$ and different hyperplanes $[V_5], [V'_5]\in Y_{A^\bot}^2$ (the other cases are more difficult).\ Set $V_4\coloneqq V_5\cap V'_5$ and restrict the diagrams
\begin{equation*}
\xymatrix@R=5mm{
  & \P(\cU_X)  \ar[dr]^-{\rho} \ar[dl] & \\
  X&&\P(V_5)
}
\qquad
\xymatrix@R=5mm{
  & \P(\cU_{X'})  \ar[dr]^-{\rho'} \ar[dl] & \\
  X'&&\P(V'_5)
}
\end{equation*}
to
\begin{equation*}
\xymatrix@R=5mm{
  & \widetilde X  \ar[dr]^-{\tilde\rho} \ar[dl]_-{\eps} &&\widetilde X'\ar[dl]_-{\tilde\rho'} \ar[dr]^-{\eps'}  \\
  X&&\P(V_4)&&X',
}
\end{equation*}
where $\widetilde X\coloneqq  \rho^{-1}(\P(V_4))$ and $ \widetilde X' \coloneqq  \rho^{\prime-1}(\P(V_4))$.\ One can show that $\eps$ and $\eps'$ are birational maps.\ Moreover, by Proposition~\ref{prop31}, $\tilde\rho$ is, outside of $(\Sigma_1(X)\cup Y^{\ge 2}_A)\cap \P(V_4)$,\footnote{By \cite[Lemma~B.6]{dk1}, this set has dimension $\le 1$.} a double cover branched along the sextic surface $Y_A\cap \P(V_4)$.\ It follows that above the complement of  $(\Sigma_1(X)\cup \Sigma_1(X')\cup \cup Y^{\ge 2}_A)\cap \P(V_4)$, the morphisms $\tilde\rho$ and $\tilde\rho'$ are double covers branched along {\em the same} sextic surface.\ The schemes~$\widetilde X$ and  $ \widetilde X'$ are therefore birationally isomorphic, and so are the varieties~$X$ and~$X'$.

A similar proof works in case $X$ and $X'$ are   ordinary GM fourfolds associated with the same Lagrangian subspace $A$ and different hyperplanes $[V_5], [V'_5]\in Y_{A^\bot}^1$: with the same notation,~$\tilde\rho$ and~$\tilde\rho'$ are, outside a subvariety of $\P(V_4)$ of dimension $\le 1$,  conic bundles with the same discriminant $Y_A\cap \P(V_4)$.\ Over a general point of every component of that surface, the fiber is the union of two distinct lines, and the double cover of $Y_A\cap \P(V_4)$ that this defines is the same for $X$ and for $X'$: one can show it is induced by restriction of the canonical double cover $\tY_A\to Y_A$.\ A general argument based on Brauer groups (see~\cite{ct}) then shows that the   fields of functions of~$\widetilde X$ and $\widetilde X'$ are the same extension of the field of functions of $\P(V_4)$.\ Therefore, $\widetilde X$ and $\widetilde X'$ are birationally isomorphic, and so are $X$ and $X'$.
\end{proof}

In \cite[Section~7]{dim}, we describe explicit birational maps between ordinary GM threefolds associated with the same Lagrangian (``conic transformations'') or with dual Lagrangians (``line transformations'').

\subsection{Rationality}

We first consider GM varieties of dimension $n\in\{5,6\}$.

\begin{prop}\label{p33}
All GM varieties of dimension $n\in\{5,6\}$ are rational.
\end{prop}

\begin{proof}
Let us do the proof for an ordinary GM fivefold $X$, corresponding to a Lagrangian $A\subset \bw3V_6$ with no decomposable vectors and a hyperplane $[V_5]\in \P(V_6)\moins Y_{A^\bot}$.\ First, one can show that the surface $Y_A^{\ge 2}$ is not contained in a hyperplane, so we can choose $[v]\in Y_A^2\moins \P(V_5)$.\ By Theorem~\ref{th22}, the corresponding 8 dimensional nonPl\"ucker quadric $Q(v)$ has corank 2 and $X=\Gr(2,V_5)\cap Q(v)$.\ One shows that for a general $6$ dimensional $\bq(v)$-isotropic subspace $I $, the intersection $S\coloneqq \P(I)\cap X=\P(I)\cap \Gr(2,V_5)$   is a smooth quintic del Pezzo surface.\ Consider the linear projection
\begin{equation*}
\pi_{S} \colon X \dashrightarrow \P^{3} 
\end{equation*}
from  $\P(I)$.\ A general fiber   is also a quintic del Pezzo surface: indeed, the intersection of a general~$\P^6$ containing $\P(I) = \P^5$ with $Q(v)$ is a union  $\P(I)\cup\P(I')$, where $I'$ is another Lagrangian subspace, and its intersection with $X$ is
  the union of $S$ and another smooth quintic del Pezzo surface $\P(I')\cap \Gr(2,V_5)$.

Therefore, the   field of rational functions of $X$ is the field of rational functions of 
a smooth quintic del Pezzo surface defined over the field of rational functions on $\P^3$.\
But a smooth quintic del Pezzo surface is rational over any field by a theorem of Enriques
(see~\cite{sb92}).
\end{proof}

The situation for GM threefolds is the following.\ First, any  GM threefold $X$ is unirational (there is a degree 2 rational map $\P^3\dra X$).\ However, using the Clemens--Griffiths criterion and a degeneration argument, one obtains that a general GM threefold is {\em not rational.}\footnote{An alternate argument can be found in \cite{demo}, where  an explicit   irrational GM threefold is constructed.}\ People believe that this should be true for {\em all}   GM threefolds, but lacking a description of the singular locus of the theta divisor of the intermediate Jacobian, this question remains open.

Similarly, any   GM fourfold $X$ is unirational (there is a degree 2 rational map $\P^4\dra X$; see \cite[Proposition~3.1]{dim2}).\ The question of their rationality is not settled and is very much like the case of cubic fourfolds in $\P^5$: some   rational examples are known (see below) and it is expected that a very general GM fourfold is not rational, but not a single example is known.\ Note that by Theorem~\ref{th32}, the rationality of a GM fourfold only depends on its associated Lagrangian.

\begin{exam}[GM fourfolds containing a quintic del Pezzo surface]\label{ex34}
Let $X$ be a     GM fourfold.\ If $Y^3_{A} \not\subset \P(V_5) $, the same argument as in the proof of Proposition~\ref{p33} shows that $X$ contains a quintic del Pezzo surface and is rational.\ More generally, using Theorem~\ref{th32}, one can show 
that~$X$ is rational as soon as $Y^3_{A} \ne \emptyset$  {(see~\cite[Lemma~4.7]{kp} for more details)}.\ This is a codimension~1 condition in the moduli space $\EPW$ for Lagrangians with no decomposable vectors: the set
\begin{equation}\label{defd}
\Delta \coloneqq  \{ [A] \in \LGr(\bw3V_6)\mid Y^3_{A} \ne \emptyset \}
\end{equation}
is a $\PGL(V_6)$-invariant irreducible hypersurface.\ Likewise, this is a codimension 1 condition in the moduli space $\bcM_4$ for GM fourfolds.
\end{exam}

\begin{exam}[GM fourfolds containing a $\sigma$-plane]\label{ex35}
A $\sigma$-plane is a plane of the type $P=\P(V_1\wedge V_4)\subset \P(\bw2V_5)$; it is contained in $\Gr(2,V_5)$.\
Let $X$ be a     GM fourfold, corresponding to a Lagrangian $A\subset \bw3V_6$ with no decomposable vectors and a hyperplane $[V_5]\in    Y^{ 1}_{A^\bot}$.\ One can show that $X$ contains a $\sigma$-plane if and only if $Y^{ 3}_{A}\cap\P(V_5)\ne\vide$.\ In particular, $[A]$ is   in the hypersurface $\Delta$ defined in \eqref{defd}; if $[A]$ is general in $\Delta$, the scheme $Y^{\ge 3}_{A}$ consists of a single point~$[v]$ and one needs $[v]\in \P(V_5)$, or equivalently, $[V_5]\in ( Y^{ 1}_{A^\bot})\cap \P(v^\bot)$.\ So this is a codimension 1 condition on $A$ in $\EPW$  and a codimension 2 condition on $X$ in $\bcM_4$.

It is easy to see directly that $X$ is rational: the image of the   linear projection
$\pi_{P} \colon X \dashrightarrow \P^{5} 
$
from  $P$ is a smooth quadric $Q\subset \P^5$ and the induced morphism $\Bl_PX\to Q$ is the blowup of a smooth degree 9 surface $\widetilde S\subset Q$ (itself the blowup of a degree 10 K3 surface $S\subset \P^6$ at a point).\ This was known to Roth (\cite[Section~4]{rot}) and Prokhorov (\cite[Section~3]{pro}).
\end{exam}

\begin{exam}[Other examples of rational GM fourfolds]\label{ex36}
The image of $\P^2$ by   the linear system of quartic curves through three simple points and one double point in general position is a smooth surface of degree $9$ in $\P^8$.\ It was proved in \cite{hs} that any   GM fourfold containing such a surface is rational.\ This is a codimension 1  condition  in~$\bcM_4$.
\end{exam}

\subsection{Linear spaces contained in GM varieties}

For any partial flag $V_1\subset V_{r+2}\subset V_5$, we consider the $r$ dimensional linear space $\P(V_1\wedge V_{r+2})\subset \Gr(2,V_5)$; it is said to be {\em of $\sigma$-type}.\ For $r\in\{1,3\}$, all $\P^r$ contained in $\Gr(2,V_5)$ are of $\sigma$-type; for $r=2$, all planes contained in~$\Gr(2,V_5)$ are either of $\sigma$-type (called $\sigma$-planes, and already considered in Example~\ref{ex35}) or of the form~$\Gr(2,V_3)$ (called $\tau$-planes).

Let $X$ be an ordinary GM variety and let $F_r(X)$ be the Hilbert scheme of $r$ dimensional linear spaces  of $\sigma$-type contained in $X$.\ There is a well defined morphism
\begin{eqnarray*}
\sigma\colon F_r(X)&\lra&\P(V_5)\\
{[}\P(V_1\wedge V_{r+2})]&\longmapsto&[V_1].
\end{eqnarray*}
We will relate these schemes $F_r(X)$ to the   fibration $\rho\colon \P(\cU_X)\to\P(V_5)$ defined in \eqref{rho}.

\begin{prop}\label{p36}
Let $X$ be a   ordinary GM variety of dimension $n\ge3$, with associated Lagrangian data set $(V_6,V_5,A)$.\
The map 
$\sigma\colon F_r(X) \to \P(V_5)$
lifts to an isomorphism
$$
F_r(X) \isom \Hilb^{{\P^r}}(\P(\cU_X) / \P(V_5))$$
with  the  relative Hilbert  schemes of  $r$ dimensional linear spaces contained in the fibers of the  morphism~$\rho$.
 \end{prop}

\begin{proof}
Assume that $P=\P(V_1\wedge V_{r+2})$ is contained in $X$.\ The inclusion $P\hra X$ lifts to an inclusion $P\hra \P(\cU_X)$ by mapping $x\in P$ to $(x,[V_1])$, and the image is contained in the fiber $\rho^{-1}([V_1])$.\ Conversely, any $\P^r$ contained in a fiber~$\rho^{-1}([V_1])$ projects isomorphically onto a $\P^r=\P(V_1\wedge V_{r+2})$ of $\sigma$-type contained in $X$.
 \end{proof}
 
 \begin{coro}\label{coro38}
Let $X$ be an   ordinary GM $n$fold.\ One has $F_3(X)=\vide$ and, if $n\le 3$, one has $F_2(X)=\vide$.
\end{coro}

\begin{proof}
By Proposition~\ref{p36}, any $\P^3$ contained  in $X$ is contained in a fiber $\rho^{-1}([v])$, which, by Proposition~\ref{prop31}, is a quadric of corank $\le 3$ in $\P^{n-2}$ if $v \notin \Sigma_1(X)$, or a quadric of corank $\le 2$ in~$\P^{n-1}$ if $v \in \Sigma_1(X)$.\ But no such quadric contains a $\P^3$, hence $F_3(X)=\vide$.\ The reasoning for~$F_2(X)$ is analogous.
\end{proof}

One could also deduce Corollary~\ref{coro38} from Hodge theory and the Lefschetz theorem (see Section~\ref{se41}).\ The next results are  more interesting;  we assume $Y^3_A=\vide$ to make the statement cleaner (but then we miss the case of GM fourfolds containing a $\sigma$-plane, as explained in Example~\ref{ex35}).\footnote{For a complete statement, see \cite[Theorem~4.3]{dk2}.\ For the description of the scheme of $\tau$-planes contained in a GM variety, see \cite[Theorem~4.5]{dk2}.}

 \begin{theo}[$\sigma$-planes on a GM variety]\label{th38}
Let $X$ be a     GM variety of dimension $n\in\{4,5,6\}$, with associated Lagrangian data set $(V_6,V_5,A)$.\ We assume $Y^3_A=\vide$.
\begin{itemize}
\item If $n=6$,
 there is a factorization
\begin{equation*}
\sigma\colon F_2(X) \xrightarrow{\ \tilde\sigma\ } \tY_{A,V_5} \xrightarrow{\ f_{A,V_5}\ } Y_{A}\cap\P(V_5)  \lhra \P(V_5), 
\end{equation*}
where $\tilde\sigma$ is a $\P^1$-bundle and $f_{A,V_5}$ is an irreducible ramified double cover (between threefolds).\ In particular, $F_2(X)$ is irreducible of dimension 4.
\item If $n=5$ and $X$ is ordinary,  there is a factorization
\begin{equation*}
\sigma\colon 
F_2(X) \xrightarrow{\ g_{A,V_5}\ }  Y^{2}_A\cap\P(V_5) \lhra \P(V_5), 
\end{equation*}
where $g_{A,V_5}$
 is a connected \'etale double cover.\ In particular, $F_2(X)$  is a connected curve.
\item If $n=4$, one has $F_2(X)=\vide$.
\end{itemize}
\end{theo}

\begin{proof}
This is again a direct consequence of Propositions~\ref{p36} and~\ref{prop31}.\ Let us do the case $n=5$ and $X$ ordinary, where $\Sigma_1(X)=\vide$.\ We have to look at planes contained in quadrics of corank $k\le 2$ in $\P^3$.\ They only exist if $k=2$ and there are exactly 2 of them; the conclusion follows (except for the connectedness of $F_2(X)$, which has to be shown by other means).
\end{proof}

The next theorem is proved similarly.

 \begin{theo}[Lines on a GM variety]\label{th39}
Let $X$ be a     GM variety of dimension $n\in\{3,4\}$, with associated Lagrangian data set $(V_6,V_5,A)$.\ We assume $Y^3_A=\vide$.
\begin{itemize}
\item If $n=4$, the scheme $F_1(X)$ is integral of dimension 3 and
  there is a factorization
\begin{equation*}
\sigma\colon F_1(X) \xrightarrow{\ \tilde\sigma\ } \tY_{A,V_5} \xrightarrow{\ f_{A,V_5}\ } Y_{A}\cap\P(V_5)  \lhra \P(V_5), 
\end{equation*}
where $\tilde\sigma$  is an isomorphism over  the complement of the inverse image by $f_{A,V_5}$ of the point $  \Sigma_1(X)$ of $ Y_{A}\cap\P(V_5) $  and has~$\P^1$ fibers over that set, and $f_{A,V_5}$ is an irreducible ramified double cover.\footnote{The same as in Theorem~\ref{th38}.}
\item If $n=3$, the scheme $F_1(X)$ is reduced of pure dimension 1 and  there is a factorization
\begin{equation*}
\sigma\colon 
F_1(X) \xrightarrow{\ \tilde\sigma\ }  Y^{2}_A\cap\P(V_5) \lhra \P(V_5), 
\end{equation*}
where $\tilde\sigma$ is an isomorphism over the complement of $Y^{2}_A\cap  \Sigma_1(X)$ and a double \'etale cover over 
that set.\  
\end{itemize}
\end{theo}

In the  next section,   we will identify  the double covers $g_{A,V_5}\colon F_2(X) \to Y^2_{A}\cap\P(V_5)$  and $f_{A,V_5}\colon \tY_{A,V_5} \to Y_{A}\cap\P(V_5)$  occurring in the   theorems: they   are induced by restriction to $\P(V_5)$ of  canonical double covers $g_A\colon \widetilde Y^{ 2}_A \to Y_{A}^2$ and $f_A\colon\tY_A \to Y_{A}$.

\begin{rema}\upshape\label{re311}
All GM varieties of dimension $\ge 3$ contain lines (\cite[Theorem~4.7]{dk2}).\ All GM varieties of dimension $\ge 5$ contain $\sigma$-planes and $\tau$-planes (\cite[Theorems~4.3 and~4.5]{dk2}).\ 
\end{rema}

\subsection{Double EPW sextics}\label{se34}
We start from a Lagrangian $A\subset \bw3V_6$  with no decomposable vectors and its associated EPW sextic $Y_A\subset \P(V_6)$ as defined by \eqref{yabot}.\   O'Grady (\cite{og1}) constructed a double cover
\begin{equation}\label{deffA}
f_A\colon \tY_A\coloneqq \bSpec(\cO_{Y_A}\oplus \cR)\lra Y_A,
\end{equation}
where $\cR$ is a certain reflexive self-dual rank 1 sheaf on $Y_A$,    as follows.

In the trivial symplectic vector bundle  $\cV\coloneqq \bw3V_6\otimes \cO_{\P(V_6)}$, one considers the Lagrangian subbundles~$\cA_1$, with constant fiber $A\subset \bw3V_6$, and  $\cA_2$, with fiber $v\wedge \bw2V_6$ at $[v]\in \P(V_6)$.\ With the notation of \eqref{dualYA}, the schemes $Y_A^{\ge \ell}$ are the {\em Lagrangian degeneration schemes}
$$D^\ell(\cA_1,\cA_2)\coloneqq \{[v]\in \P(V_6)\mid \dim( \cA_{1,[v]}\cap  \cA_{2,[v]})\ge\ell\}.$$
More precisely, $D^\ell(\cA_1,\cA_2)$ is defined as the   corank $\ell$ degeneracy locus of the morphism
\begin{equation*}
\omega_{\cA_1,\cA_2} \colon 
 \cA_1 \lhra \cV \xrightarrow{\ \sim\ } \cV^\vee  \twoheadrightarrow \cA_2^\vee  ,
\end{equation*}
where the middle isomorphism  is induced by the symplectic form on $\cV$.\ For example, $D^1(\cA_1,\cA_2)=Y^{\ge1}_A=Y_A$ is the scheme-theoretic support of the sheaf $\cC\coloneqq \Coker (\omega_{\cA_1,\cA_2})$.\ At a point $[v]\in Y^1_A$, the morphism $\omega_{\cA_1,\cA_2}$ induces an isomorphism
\begin{equation}\label{isoo}
 \cA_{1,[v]}/(\cA_{1,[v]}\cap  \cA_{2,[v]})\isomlra (\cA_{1,[v]}\cap  \cA_{2,[v]})^\bot\subset\cA^\vee_{2,[v]}
\end{equation}
hence the sheaf $\cC$ has rank $1$ on $Y^1_A$, with fiber $(\cA_{1,[v]}\cap  \cA_{2,[v]})^\vee$ at $[v]$.

The fiber of the dual $\cC^\vee$ at a point $[v]\in Y^1_A$ is therefore $\cA_{1,[v]}\cap  \cA_{2,[v]}$, and the map \eqref{isoo} induces an isomorphism  between $\cA_1/\cC^\vee$ and $(\cA_2/\cC^\vee)^\vee$ on~$Y^1_A$, hence between their determinants, which are $\cC^{\vee\vee}$ and $\cC^{\vee}(-6) $ (because $\cA_1$ is trivial and $\det(\cA_2)\isom \cO(6)$).\ If we set
$$\cR\coloneqq \cC^{\vee\vee}( -3),$$
 this gives an isomorphism between the invertible sheaf $\cR\vert_{Y^1_A}$ and its dual.\ Since $\cR$ and $\cR^\vee$ are reflexive and $Y_A^{\ge 2}$ has codimension $2$ in $Y_A$ (Theorem~\ref{th24}), this self-duality extends uniquely to a self-duality on $Y_A$, which we see as a symmetric multiplication map
$$\cR\otimes \cR\lra \cO_{Y_A}$$
that defines an  $ \cO_{Y_A}$-algebra structure on $ \cO_{Y_A}\oplus \cR$ hence the double cover $f_A$ in \eqref{deffA}.

The singularities of $Y_A$ and $\tY_A$ were described in \cite{og4}.

\begin{theo}[O'Grady]\label{th310}
Let  $A\subset \bw3V_6$ be a Lagrangian  with no decomposable vectors.\
The morphism $f_A\colon \tY_A\to Y_A$ is branched along the  integral normal surface~$Y^{\ge 2}_A$, which is the singular locus of $Y_A$.\ 
The fourfold~$\tY_A$, called a {\em double EPW sextic}, is irreducible and smooth outside the finite set $f_A^{-1}(Y^3_A)$.
\end{theo}

Using similar ideas (see~\cite{dk4}), one can also construct a canonical irreducible double cover
\begin{equation}\label{gA}
g_A\colon \widetilde Y^{\ge 2}_A\lra Y^{\ge 2}_A
\end{equation}
branched along the finite set $Y^3_A$, such that 
$$(g_A)_*\cO_{\widetilde Y^{\ge 2}_A}\isom \cO_{Y^{\ge 2}_A}\oplus \omega_{Y^{\ge 2}_A}(-3).$$
The singular locus of the surface $\widetilde Y^{\ge 2}_A$ is the   finite set $f_A^{-1}(Y^3_A)$.\ When $Y^3_A=\vide$, the surface $\widetilde Y^{\ge 2}_A$ is  smooth     of general type of irregularity $h^1(\widetilde Y^{ 2}_A,\cO_{\widetilde Y^2_A})=10$.

 Finally, one  can prove (see~\cite{dk4}) that  the double covers $g_{A,V_5}\colon F_2(X) \to Y^2_{A}\cap\P(V_5)$  and   $f_{A,V_5}\colon \tY_{A,V_5} \to Y_{A}\cap\P(V_5)$ that appear in Theorems~\ref{th38} and~\ref{th39}   are the restrictions to $\P(V_5)$ of the canonical double covers $g_A\colon \widetilde Y^{ 2}_A \to Y_{A}^2$ and $f_A\colon\tY_A \to Y_{A}$.

\begin{coro}\label{cor312}
Let  $A\subset \bw3V_6$ be a Lagrangian subspace with no decomposable vectors and  such that $Y_A^3=\vide$.\ Let $X$ be the ordinary GM fourfold associated with a general $[V_5]$  in $Y_{A^\bot}$.\  
\begin{itemize}
\item The threefold $\tY_{A,V_5} $ has two singular points.
\item The threefold $F_1(X)$ is   smooth irreducible   and $\tilde\sigma\colon F_1(X)\to \tY_{A,V_5}$ is a small resolution.
\end{itemize}
\end{coro}

\begin{proof}
We already mentioned that $Y_{A^\bot}$ is the projective dual of $Y_A$.\ Projective duality implies that the inverse image by $\tY_A\to Y_A\subset \P(V_6)$ of a general 
$[V_5]$  in $Y_{A^\bot}$ has at least two singular points.\ More precisely, 
the single point $[v]$ of $\Sigma_1(X)$ is in $Y^1_{A}\cap\P(V_5)$ (see Proposition~\ref{prop31}) and the singular locus of $\tY_{A,V_5} $ consists of its two preimages in $\tY_A$.\ The conclusion then follows from
Theorem~\ref{th39} and the verification via a tangent space computation that $F_1(X)$ is smooth.
\end{proof}

 \begin{coro}\label{cor3125}
Let  $A\subset \bw3V_6$ be a Lagrangian subspace with no decomposable vectors and  such that $Y_A^3=\vide$.\ Let $X$ be the ordinary GM fivefold associated with a  general hyperplane~$V_5\subset V_6$.\  

The curve $F_2(X)$ is  smooth connected  of genus 161 and $\tilde\sigma\colon F_2(X)\to Y^2_A\cap \P(V_5)$ is an \'etale double cover.
\end{coro}

\begin{proof}
Since $V_5$ is general and $Y^2_A$ is smooth, so is    $Y^2_A\cap \P(V_5)$  by  Bertini's theorem.\ The conclusion then follows from Theorem~\ref{th38} and a direct computation of the genus using the numerical invariants of the smooth surface $Y^2_A$.\end{proof}

 \begin{coro}\label{cor311}
Let  $A\subset \bw3V_6$ be a Lagrangian subspace with no decomposable vectors and  such that $Y_A^3=\vide$.\ Let $X$ be the  GM sixfold associated with a  general hyperplane~$V_5\subset V_6$.\  

The threefold $F_2(X)$ is smooth irreducible and $\tilde\sigma\colon F_2(X)\to f_A^{-1}(Y_A\cap \P(V_5))$ is a $\P^1$-fibration.
\end{coro}

\begin{proof}
Since $V_5$ is general and $\tY_A$ is smooth, so is   $\tY_{A,V_5}=f_A^{-1}(Y_A\cap \P(V_5))$  by  Bertini's theorem.\ The conclusion then follows from Theorem~\ref{th38}.\end{proof}

\begin{rema}\upshape\label{rem314}
When $X$ is a general GM threefold, $F_1(X)$ is a smooth connected curve of genus~$71$ (\cite[Theorem~4.2.7]{ip}), $Y^{2}_A\cap\P(V_5)$ is an irreducible curve with 10 nodes as singularities, and $\tilde\sigma\colon F_1(X) \to Y^{2}_A\cap\P(V_5)$ is its normalization.
\end{rema}

 \section{Periods of Gushel--Mukai varieties}\label{se4}

 \subsection{Cohomology and period maps of GM varieties}\label{se41}
With some work, the Hodge numbers of GM varieties can   be computed.\ The references for the next proposition are  \cite{dk0} and \cite[Propositions~3.1 and~3.4]{dk2}).

\begin{prop}\label{hn}
The integral cohomology of a     GM variety  of dimension $n$ is torsionfree and its
  Hodge diamond    is  
\begin{equation*} 
\begin{array}{cccccc}
(n=1) & (n=2) & (n = 3) & (n=4) & (n=5) & (n = 6) \\[1ex]  
\begin{smallmatrix}
 &1 \\
6 && 6 \\
& 1
\end{smallmatrix} &
\begin{smallmatrix}
&& 1 \\
 & 0&&0 \\
1&& 20&&1  \\
 & 0&&0 \\
&& 1 
\end{smallmatrix} &
\begin{smallmatrix}
 &&& 1 \\
  && 0 && 0  \\
&0&& 1 &&0 \\
 0&& 10 && 10 &&0  \\
&0&& 1 &&0 \\
  && 0 && 0  \\
 &&& 1 
\end{smallmatrix} &
\begin{smallmatrix}
&& &&&& 1 \\
&& &&& 0 && 0  \\
& &&&0&& 1 &&0 \\
&& &0&& 0 && 0 &&0  \\
&&0 &&1&& 22 &&1&&0 	\\
&& &0&& 0 && 0 &&0  \\
& &&&0&& 1 &&0 \\
&& &&& 0 && 0  \\
&& &&&& 1 
\end{smallmatrix} &
\begin{smallmatrix}
&& &&&& 1 \\
&& &&& 0 && 0  \\
& &&&0&& 1 &&0 \\
&& &0&& 0 && 0 &&0  \\
&&0 &&0&& 2 &&0&&0 \\
& 0&& 0 && 10 &&10&&0&&0  \\
&&0 &&0&& 2 &&0&&0 \\
&& &0&& 0 && 0 &&0  \\
& &&&0&& 1 &&0 \\
&& &&& 0 && 0  \\
&& &&&& 1 
\end{smallmatrix} &
\begin{smallmatrix}
&& &&&& 1 \\
&& &&& 0 && 0  \\
& &&&0&& 1 &&0 \\
&& &0&& 0 && 0 &&0  \\
&&0 &&0&& 2 &&0&&0 \\
& 0&& 0 && 0 &&0&&0&&0\\
  0&& 0 && 1 &&22&&1&&0&&0  \\
& 0&& 0 && 0 &&0&&0&&0\\
&&0 &&0&& 2 &&0&&0 \\
&& &0&& 0 && 0 &&0  \\
& &&&0&& 1 &&0 \\
&& &&& 0 && 0  \\
&& &&&& 1 
\end{smallmatrix} 
\end{array}
\end{equation*}
 \end{prop}

In particular, only the  middle cohomology $H^n(X;\Z)$ is interesting: in other degrees, it is induced from the cohomology of $\Gr(2,V_5)$ via the Gushel map $\gamma\colon X\to \Gr(2,V_5)$.\ We define the {\em vanishing cohomology} as
 $$H^n(X;\Z)_{\rm van}\coloneqq \bigl(\gamma^*H^n(\Gr(2,V_5);\Z)\bigr)^\bot \subset H^n(X;\Z).$$
  The  Hodge numbers for the vanishing cohomology are therefore
  \begin{equation*} 
\begin{array}{ccccccccccc}
(n=1)&&(n=2)&&(n = 3)&&(n=4)&&(n=5)&&(n = 6) \\[1ex]  
\begin{smallmatrix}
6 && 6 
\end{smallmatrix} &&
\begin{smallmatrix}
1&& 19&&1 
\end{smallmatrix} &&
\begin{smallmatrix}
 0&& 10 && 10 &&0 
\end{smallmatrix} &&
\begin{smallmatrix}
0 &&1&& 20 &&1&&0 
\end{smallmatrix} &&
\begin{smallmatrix}
 0&& 0 && 10 &&10&&0&&0 
\end{smallmatrix} &&
\begin{smallmatrix}
  0&& 0 && 1 &&20&&1&&0&&0  
\end{smallmatrix} 
\end{array}
\end{equation*}
  In other words, 
  \begin{itemize}
\item  when $n\in\{3,5\}$, this Hodge structure has weight 1 and there is a 10 dimensional principally polarized intermediate Jacobian 
$$J(X)\coloneqq  H^n(X,\C)/\bigl( H^{(n+1)/2,(n-1)/2}(X)+H^n(X;\Z)\bigr)$$
 and a period map
\begin{eqnarray*}
\wp_n\colon \bcM_n&\lra& \cA_{10}\\
{[}X]&\longmapsto& [J(X)],
\end{eqnarray*}
where $ \cA_{10}$ is the moduli space for principally polarized abelian varieties of dimension~$10$, a  $55$ dimensional quasiprojective variety ;
\item when $n\in\{4,6\}$, this Hodge structure is of K3 type and there is a period map
\begin{eqnarray*}
\wp_n\colon \bcM_n&\lra& \cD\\
{[}X]&\longmapsto& [H^{n/2+1,n/2-1}(X)],
\end{eqnarray*}
where $\cD$ is a  $20$ dimensional quasiprojective variety 
 (the same for both $n=4
$ and $n=6$) which will be defined in Section~\ref{se45}.
\end{itemize}
None of these period maps  
 $$
\xymatrix
@C=20pt
@R=5pt
{\hbox{dim.}&&&&\hbox{dim.}\\
22& \bcM_3\ar[drr]^-{\wp_3}\\
24&\bcM_4\ar[drr]^(.34){\wp_4}&&\cA_{10}&55\\
25& \bcM_5\ar[urr]_(.34){\wp_5}&&\cD&20\\
25& \bcM_6\ar[urr]_-{\wp_6}
}
$$
 are injective.\ They are dominant when $n$ is even.\ To analyze them further, we need to go back to the double EPW sextics introduced in Section~\ref{se34}.  
 
 \begin{rema}[The Hodge conjectures for GM varieties] 
The {\em rational} Hodge conjecture holds for all GM varieties   of dimensions $n\ge3$ (\cite[Theorem~5.1]{fumo}).


The {\em integral} Hodge conjecture holds for all  GM varieties of dimension $n\in\{3,5\}$, because the Hodge classes all come from the Grassmannian.\ It also holds for   all  GM fourfolds (see \cite[Corollary~1.2]{perr}, \cite[Remark~1.6]{ppz}).\ It seems to be still unknown for $(3,3)$-classes on GM sixfolds (see \cite[Corollary~8.4]{perr} for a partial result).
%
\end{rema}

 \subsection{Double EPW sextics and hyper-K\"ahler fourfolds}
 
 The reason why O'Grady made the construction (explained in Section~\ref{se34}) of double EPW sextics is that they provide examples of hyper-K\"ahler fourfolds.
 
 \begin{defi}
A hyper-K\"ahler variety is a smooth, compact, simply connected, K\"ahler variety  whose space of holomorphic 2 forms is generated by a symplectic form.
\end{defi}

Hyper-K\"ahler varieties of dimension 2 are K3 surfaces.\ If $S$ is a K3 surface, it was shown by Beauville that the punctual Douady space $S^{[n]}\coloneqq \Hilb^n(S)$ parametrizing length $n$ subschemes of $S$ is a hyper-K\"ahler variety of dimension $2n$.\ The main result of \cite{og1} is the following (the double EPW sextic $\tY_A$ was defined in Theorem~\ref{th310}).

\begin{theo}
Let  $A\subset \bw3V_6$ be a Lagrangian  with no decomposable vectors and such that $Y^3_A=\vide$.\ The double EPW sextic $\tY_A$  is a   hyper-K\"ahler fourfold.
\end{theo}

Moreover, O'Grady proved that $\tY_A$  is a deformation of the Douady square $S^{[2]}$ of a K3 surface $S$, so that its Hodge numbers are known by a computation of Beauville: for $H^2( \tY_A)$, they are
  \begin{equation*} 
\begin{matrix}
1&& 21&&1  
\end{matrix} .
\end{equation*}
 Beauville also defined an integral, nondegenerate quadratic form $q_{BB}$ on $H^2( \tY_A;\Z)$.\ The class $h_A\coloneqq f_A^*\cO_{\P(V_6)}(1)$ is ample on $\tY_A$ with $q_{BB}(h_A)=2$, and we define the {\em primitive cohomology} as
 $$H^2(\tY_A;\Z)_0\coloneqq h_A^\bot \subset H^2(\tY_A;\Z).$$

The primitive cohomology gives rise to a period map
\begin{eqnarray*}
\wp\colon \EPW&\lra& \cD\\
{[}A]&\longmapsto& [H^{2,0}(\tY_A)],
\end{eqnarray*}
where $\cD$ is the  period domain already mentioned   in Section~\ref{se41} and which will be defined in Section~\ref{se45} (strictly speaking, this map is only defined on the complement of the hypersurface~$\Delta$ defined in~\eqref{defd}---where $\tY_A $ is smooth---but O'Grady showed in \cite{og4}  that it extends over~$\Delta$).\ By   Verbitsky's global Torelli Theorem and Markman's monodromy results, and \cite[Section~5.7]{og4}, {\em the map $\wp$ is an open embedding.}

  \subsection{Factorization of the period maps of GM varieties}

 We explain in this section that the period map $\wp_n\colon \bcM_n\to \textnormal{($\cA_{10}$ ou $\cD$)}$  for GM $n$folds, defined in Section~\ref{se41}, factors through the  surjective morphism
$$\pi_n\colon \bcM_n\lra \EPW$$
defined in Theorem~\ref{th26} that sends a GM variety $X$ of dimension $n$  to its associated Lagrangian $A\subset \bw3V_6$.

For that, we relate the Hodge structure~$H^n(X)_{\rm van}$   to the Hodge structure   $H^2(\tY_A)_0$ when $n\in\{4,6\}$, and to the Hodge structure   $H^1(\widetilde Y^2_A)$ when $n\in\{3,5\}$.

 \begin{theo}[\cite{dk2,dk5}] \label{th44}
 Let $X$ be a     GM variety of dimension $n\in\{3,4,5,6\}$ with associated Lagrangian $A$.\ We assume $Y^3_A=\vide$.\footnote{This hypothesis ensures that both   $\tY_A$ and $\widetilde Y^2_A$  are smooth.}
 
  \noindent  {\rm(a)} When  $n\in\{4,6\}$, there is an isomorphism of polarized Hodge structures
 $$(H^n(X;\Z)_{\rm van},\smile)\isomlra (H^2(\tY_A;\Z)_0,(-1)^{n/2-1}q_{BB}).$$

\noindent{\rm(b)} There is a canonical principal polarization on the abelian variety $ \Alb(\widetilde Y^2_A)$ and, when $n\in\{3,5\}$, there is an isomorphism  
 $$J(X)\isomlra \Alb(\widetilde Y^2_A)$$
 of principally polarized abelian varieties. 
\end{theo}

 When $n\in\{4,6\}$, the period map $\wp_n\colon \bcM_n\to \cD$  therefore  factors as
\begin{equation}\label{pmapse}
  \vcenter{
\xymatrix
@C=15pt
@R=5pt
@M=5pt
{\hbox{dim.}&&&&&&\hbox{dim.}\\
24&\bcM_4\ar@{->>}[drr]^-{\pi_4}\\
&&&\EPW\ar@{^{(}->}[rr]^-\wp&&\cD&20\\
25&\bcM_6\ar@{->>}[urr]_-{\pi_6}\\
}}
\end{equation}
where $\wp$ is  the (extended) period map  for double EPW sextics.\  When $n\in\{3,5\}$, the period map $\wp_n\colon \bcM_n\to \cA_{10}$  factors as 
\begin{equation}\label{pmapso}
  \vcenter{
\xymatrix
@C=17pt
@R=5pt
@M=5pt
{\hbox{dim.}&&&&&&\hbox{dim.}\\
22&\bcM_3\ar@{->>}[drr]^-{\pi_3}\\
&&&\EPW\ar[r]&\EPW/r\ar[r]^-{\mathsf{q}}&\cA_{10}&55\\
25& \bcM_5\ar@{->>}[urr]_-{\pi_5}\\
}}
\end{equation}
where $r$ is the (nontrivial) duality involution   of $\EPW$ (see \cite[Lemma~6.1]{dk5}).\ The morphism~${\mathsf{q}}$ is known to be unramified (\cite{dim}) and   is expected to be (generically) injective.

 \subsection{Sketches of proofs}
 
 The standard argument for this kind of results goes back to Clemens--Griffiths (\cite{cg}) and Beauville--Donagi (\cite{bedo}), who treated the case of a smooth cubic fourfold $W\subset \P(V_6)$ and its (smooth) hyper-K\"ahler fourfold of lines $F_1(W)\subset \Gr(2,V_6)$.\ In that case, there is an incidence diagram
 $$\xymatrix@R=5mm{
& I \ar[dl]_-q \ar[dr] ^-p\\
W && F_1(W),
}$$
where $p$ is the universal line (a $\P^1$-bundle) and $q$ is dominant and generically finite.\ Beauville and Donagi prove that the Abel--Jacobi map $p_*q^*\colon H^4(W;\Z)\to H^2(F_1(W);\Z)$ is an isomorphism of Hodge structures which induces an isometry between the primitive cohomologies \mbox{$(H^4(W;\Z)_0,\smile)$} and $( H^2(F_1(W);\Z)_0,-q_{BB})$.\ The two  ingredients used in this argument are:
\begin{itemize}
\item $F_1(W)$ parametrizes curves on $W$, so there is a correspondence between $W$ and $F_1(W)$;
\item $F_1(W)$ and $I$ are smooth, so one can define $p_*$ in singular cohomology.
\end{itemize}
We follow a similar approach, using Hilbert schemes of lines or $\sigma$-planes on GM varieties, depending on the dimension.\ In what follows, $X$ is a GM $n$fold with associated Lagrangian $A$.\ We assume $Y_A^3=\vide$.
\medskip

\noindent{\bf Dimension $4$.}
By Corollary~\ref{cor312},   for $[V_5]$ general in $Y_{A^\bot}$, the scheme of lines $F_1(X)$ is a smooth threefold  and $\tilde\sigma\colon F_1(X) \to f_A^{-1}(Y_A\cap \P(V_5))  $ is a small resolution.\ One shows that    the induced composition
$$a\colon H^2(\tY_A;\Z)\isomlra H^2(f_A^{-1}(Y_A\cap \P(V_5));\Z) \xrightarrow{\ \tilde \sigma^*\ } H^2(F_1(X);\Z)
$$
(where the first map is an isomorphism by the Lefschetz theorem) is injective.

 The Abel--Jacobi map $p_*q^*\colon H^4(X;\Z)_{\rm van}\to H^2(F_1(X);\Z)$ is also injective (as in \cite{bedo}) and induces an anti-isometry between $H^4(X;\Z)_{\rm van}$ and $a(H^2(\tY_A;\Z)_0 )\subset  H^2(F_1(X);\Z)$ (see \cite[Section~5.1]{dk2} for more details).

\medskip

\noindent{\bf Dimension $6$.}
For a general GM sixfold $X$, the proof is similar: one uses instead the smooth fourfold $F_2(X)$ parametrizing $\sigma$-planes contained  in $X$, and Corollary~\ref{cor311} (see \cite[Section~5.2]{dk2} for more details).

\medskip

\noindent{\bf Dimension $5$.}
By Corollary~\ref{cor3125},   for $[V_5]$ general in $Y_{A^\bot}$, the
curve  $F_2(X)$ is a connected double \'etale cover of a general genus $81$ hyperplane section of the smooth surface $Y^2_A\subset \P(V_6)$.\ A generalization of an old argument of Clemens (written by Tjurin in~\cite{tju}) shows that the corresponding Abel--Jacobi map $q_*p^*\colon H_1(F_2(X);\Z) \to H_5(X;\Z)  $ in homology  is surjective.\ It induces a surjective morphism
$$a\colon J(F_2(X))\lra J(X)$$
with connected kernel.\ By the Lefschetz theorem, there is  another surjective morphism
$$b\colon J(F_2(X))\lra \Alb (\widetilde Y^2_A )$$
with connected kernel.\ We want to show that the morphisms $a$ and $b$ are the same.

We use a  trick: it follows from Deligne--Picard--Lefschetz theory that since the surface $Y_A^2$ is regular, for a very general choice of hyperplane $V_5$, the Jacobian  $J(Y_A^2\cap \P(V_5))$ is simple (of dimension 81).\ It is therefore contracted by both $a$ and $b$, and this induces surjective morphisms
$$a'\colon \Prym\lra J(X) \qquad{\rm and}\qquad b'\colon \Prym\lra \Alb (\widetilde Y^2_A )$$
with connected kernels.\footnote{Here, $\Prym\coloneqq J(F_2(X))/ J(Y_A^2\cap \P(V_5))$ is the 80 dimensional  Prym variety associated with the  double \'etale cover of curves $ F_2(X)\to Y_A^2\cap \P(V_5)$.}\ Using again  monodromy arguments, one shows that the kernel of $b'$ is simple (of dimension 70).\ It is therefore contracted by $a'$, and this induces an isomorphism  
$$   \Alb (\widetilde Y^2_A )\isomlra J(X).$$
(See  \cite[Section~5]{dk5} for more details.)
 
\medskip

\noindent{\bf Dimension $3$.}
 By Remark~\ref{rem314}, lines on a  general GM threefold $X$ are parametrized by a smooth connected curve $F_1(X)$ of genus 71 which is the normalization of the singular curve $Y^2_A\cap \P(V_5)$ (the hyperplane $V_5$ is not general anymore), but we were unable to relate this curve to the surface~$\widetilde Y^2_A $.
 
 However, it was proved by Logachev in \cite{lo} that the Hilbert scheme  of conics contained in a general GM threefold  $X$ is the blowup of the smooth surface $\widetilde  Y^2_{A^\bot}$ at a point.\footnote{For a complete description of  the Hilbert scheme  of conics contained in any GM threefold, see \cite[Theorems~1.1 and~7.3]{dk6}.}\ This   gives an Abel--Jacobi map $\Alb (\widetilde Y^2_{A^\bot})\to J(X)$ which should be an isomorphism.\footnote{Note that by \cite[Theorem 1.1]{dk5}, the smooth surfaces $\widetilde  Y^2_{A}$ and $\widetilde  Y^2_{A^\bot}$ have isomorphic  Albanese varieties, although they are in general not isomorphic.}

 We   proceed differently in \cite{dk5} and prove instead that the Abel--Jacobi map associated with  a family of rational quartic curves parametrized by the surface 
 $\widetilde Y^2_A$ gives the desired isomorphism $\Alb (\widetilde Y^2_A)\isomto J(X)$  (see \cite[Section~4]{dk5} for more details).
 
 \subsection{Hodge special GM fourfolds}\label{se45}
 
 Let us first explain what the period domain $\cD$ is.\ For any~GM fourfold $X$, the unimodular cohomology lattice $(H^4(X;\Z),\smile)$ is isomorphic to $(1)^{\oplus 22} \oplus (-1)^{\oplus 2} $ and the vanishing cohomology lattice $(H^4(X;\Z)_{\rm van},\smile)$ to the 
  even lattice
  $$\Lambda\coloneqq E_8^{\oplus 2}\oplus \begin{pmatrix}0&1\\1&0\end{pmatrix}^{\oplus 2}\oplus \begin{pmatrix}2&0\\0&2\end{pmatrix},
 $$
  where $E_8$ is the  rank 8 positive definite even lattice.\ It has signature $(20,2)$.
 
 Similarly, for any Lagrangian $A$, the cohomology lattice $(H^2(\tY_A;\Z),q_{BB})$ is isomorphic to the lattice
   $$E_8(-1)^{\oplus 2}\oplus \begin{pmatrix}0&1\\1&0\end{pmatrix}^{\oplus 3}\oplus (-2),
 $$
and the primitive cohomology lattice $(H^2(\tY_A;\Z)_0,q_{BB})$ is isomorphic to $\Lambda(-1)$ (in accordance with Theorem~\ref{th44}(a)).

The   manifold
 $$\Omega\coloneqq \{\omega\in \P(\Lambda\otimes\C)\mid   (\omega\cdot \omega) = 0\ , (\omega \cdot \bar \omega) < 0\} $$
is a homogeneous space with two  components, $\Omega^+$ and $\Omega^-$, 
 both isomorphic to the 20 dimensional open complex manifold $\SO_0(20,2)/\SO(20)\times \SO(2)$, a bounded symmetric domain of type IV acted on by the isometry group~$  O(\Lambda)$ of the lattice $\Lambda$.\ The quotient
 $$\cD\coloneqq   \widetilde  O(\Lambda)\backslash \Omega^+,$$
where  $\widetilde  O(\Lambda)\subset O(\Lambda)$ is a subgroup of index 2 called the {\em stable orthogonal group,} has the structure of an irreducible quasiprojective variety of dimension~20  (see \cite[Section~5]{dim2} for more details).\ It is a period domain for both GM fourfolds and double EPW sextics.

The domain $\cD$ carries a nontrivial canonical involution~$r_\cD$ corresponding to the double cover $\cD\to  O(\Lambda)\backslash \Omega^+$.\ If $\wp\colon\EPW\to \cD$
is the (extended) period map
 for double EPW sextics, it is related to the duality involution $r$ of $\EPW$ defined in Section~\ref{se25} by the relation
 $$\wp\circ r=r_\cD\circ \wp$$
proved in \cite{og2}.

The fact that the period map
$\wp$
 is dominant easily implies that the Picard group of a very general EPW sextic $\tY_A$ is generated by the class of the polarization $h_A=f_A^*\cO_{\P( V_6)}(1)$.\ Indeed, for $ H^{1,1}(\tY_A)\cap H^2(\tY_A;\Z)_0$ to be nonzero, the corresponding period must be in one of the (countably many) loci    $\alpha^\bot \cap\cD$, for some nonzero $\alpha \in \Lambda  $.\ We label these loci as $\cD_d$, where~$d$ is the discriminant of the lattice $\alpha^\bot \subset \Lambda$, called the {\em nonspecial lattice.}\ The Picard group of a double EPW sextic whose period point is a very general point of $\cD_d$ 
 (more precisely,   it is not in any other~$\cD_{d'}$, for $d'\ne d$) has rank $2$: it is the saturation of the subgroup $\Z h_A\oplus \Z\alpha$ of~$H^2(\tY_A;\Z)$.

 The following result is \cite[Lemma~6.1 and Corollary 6.3]{dim2}.
 
 \begin{prop}
The locus $\cD_d$ is nonempty if and only if $d>0$ and
\begin{itemize}
\item either  $d\equiv 0 \pmod{4}$, in which case $\cD_d\subset \cD$ is an   irreducible hypersurface;
\item or $d\equiv 2 \pmod{8}$,  in which case $\cD_d\subset \cD$ is the union of two irreducible hypersurfaces $\cD'_d$ and $\cD''_d$, which are interchanged by the involution $r_\cD$.
\end{itemize}
\end{prop}

It is also known that
\begin{itemize}
\item the image of the (extended) period map $\wp\colon\EPW\to \cD$ does not meet the hypersurface $\cD_2\cup\cD_4\cup\cD_8$  (\cite[Theorem~1.3]{og6}, \cite[Example~6.3]{dm}) and is conjectured to be equal to its complement $\cD\moins(\cD_2\cup\cD_4\cup\cD_8)$;
%
\item the image of the hypersurface $\Delta$ defined in \eqref{defd}   is the hypersurface  $\cD''_{10}$ (see Example~\ref{ex35a} below for a proof);

\item the rational map $ \bEPW\dra \cD$ induced by $\wp$ is defined at  general points of the hypersurface $\bEPW \moins \EPW$ corresponding to Lagrangians with decomposable vectors and the image of this hypersurface is the hypersurface  $\cD_{8}$ (see~\cite[Remark~5.29]{dk2}).
\end{itemize}

This whole discussion translates in terms of GM fourfolds.\footnote{It also applies to GM sixfolds, but we will not discuss that aspect here.}\ We say that a GM fourfold~$X$ is {\em Hodge-special} if its period is in $\bigcup_{d}\cD_d$.\ For $X$ nonHodge-special, we have 
$$ H^{2,2}(X) \cap H^4(X;\Z)=\gamma^*H^4(\Gr(2,V_5);\Z)  $$
(a lattice of rank~2)  and, if $X$ is very general in a hypersurface $\wp_4^{-1}(\cD_d)$, the lattice  $ H^{2,2}(X) \cap H^4(X;\Z)$ of Hodge $(2,2)$ classes has rank~$3$ and discriminant~$d$.

Because of diagram~\eqref{pmapse}, the images of the period maps $\wp$ and
$$\wp_4\colon \bcM_4\lra \cD$$
are the same.\ The examples below, which describe geometrically certain families of GM fourfolds whose periods dominate the hypersurfaces $ \cD''_{10}$, $ \cD'_{10}$, and $ \cD_{20}$ of $\cD$, are taken from~\cite{dim2} and~\cite[Remark~5.29]{dk2}.

\begin{exam}[GM fourfolds containing a $\sigma$-plane]\label{ex35a}
This is a continuation of Example~\ref{ex35}.\ Let 
  $P=\P(V_1\wedge V_4)$ be a $\sigma$-plane contained in a GM fourfold $X$.\ Its class in $\Gr(2,V_5)$ is the Schubert class $\sigma_{3,1}$ and one computes  (see  \cite[Section~7.1]{dim2}) the intersection numbers 
  $$(\gamma^*\sigma_{1,1}\cdot P)_X=(\sigma_{1,1}\cdot P)_{\Gr(2,V_5)}=0 \ ,\ \   
  (\gamma^*\sigma_2\cdot P)_X=(\sigma_2\cdot P)_{\Gr(2,V_5)}=1\ ,\ \ 
  (P\cdot P)_X=3     
  $$ of classes  in $H^4(X;\Z)$.\ It follows that the intersection form on the rank 3 sublattice 
  $$K\coloneqq \langle\gamma^*H^4(\Gr(2,V_5);\Z),[P]\rangle=\langle\gamma^*\sigma_{1,1},\gamma^*\sigma_2 ,[P]\rangle$$ is given by the matrix
  $$\begin{pmatrix}
  2&2&0 \\
  2&4&1 \\
  0&1&3
  \end{pmatrix},
  $$
  whose determinant is $10$.\ The orthogonal of $[P]$ in $H^4(X;\Z)_{\rm van}$ is the same as the orthogonal of the primitive sublattice~$K$ in the unimodular lattice $H^4(X;\Z)$, and they have the same discriminant as $K$.\ The period of $X$ therefore belongs to the hypersurface $\cD_{10}$, and more precisely to the component $\cD''_{10}$ (see the comment on how we label the two components after \cite[Corollary~6.3]{dim2}).\ Since the periods of these fourfolds are dense in the image of $\Delta$ in $\cD$ (see Example~\ref{ex35}), we see that this image is $\cD''_{10}$.
  \end{exam}

\begin{exam}[GM fourfolds containing a quintic del Pezzo surface]\label{ex34a}
This is a continuation of Example~\ref{ex34}.\ We consider  
quintic del Pezzo surfaces obtained as the intersection of $\Gr(2,V_5)$ with a $\P^5$; their class is $\sigma_1^4=3\sigma_{3,1} + 2\sigma_{2,2}$ in $\Gr(2,V_5)$.\ Let $X$ be a GM fourfold containing such a surface $S$.\ One  computes  (see \cite[Section~7.5]{dim2}) the intersection numbers 
  $$(\gamma^*\sigma_{1,1}\cdot S)_X=(\sigma_{1,1}\cdot \sigma_1^4)_{\Gr(2,V_5)}=2 \ ,\ \   
  (\gamma^*\sigma_2\cdot S)_X=(\sigma_2\cdot \sigma_1^4)_{\Gr(2,V_5)}=3\ ,\ \ 
  (S\cdot S)_X=5     
  $$ of classes  in $H^4(X;\Z)$.\ The intersection form on the rank 3 sublattice 
  $ \langle\gamma^*\sigma_{1,1},\gamma^*\sigma_2 ,[S]\rangle$ is therefore given by the matrix
  $$\begin{pmatrix}
  2&2&2 \\
  2&4&3 \\
  2&3&5
  \end{pmatrix},
  $$
  whose determinant is $10$.\ Again, the period of $X$   belongs to the hypersurface $\cD''_{10}$ (this was expected since the associated Lagrangians as the same as the Lagrangians associated with GM fourfolds containing a $\sigma$-plane).\ More precisely, GM fourfolds containing a quintic del Pezzo surface form a dense  subset of $\wp_4^{-1}(\cD''_{10})$.
\end{exam}

\begin{exam}[GM fourfolds containing a $\tau$-quadric surface] \label{qua}
A $\tau$-quadric surface $Q$ in $ \Gr(2,V_5)$  is a linear section of a $\Gr(2,V_4)$ for some $V_4\subset V_5$; its class  in $\Gr(2,V_5)$ is $\sigma_1^2\cdot \sigma_{1,1}=\sigma_{3,1}+\sigma_{2,2}$.\ If $X$ is a GM fourfold containing $Q$, one computes as above (see \cite[Section~7.3]{dim2})  that the intersection form on the rank 3 sublattice 
  $\langle\gamma^*\sigma_{1,1},\gamma^*\sigma_2 ,[Q]\rangle$ is given by the matrix
  $$\begin{pmatrix}
  2&2&1 \\
  2&4&1 \\
  1&1&3
  \end{pmatrix},
  $$
  whose determinant is $10$.\ This time however, the period of $X$   belongs to the hypersurface $\cD'_{10}$ and GM fourfolds containing a $\tau$-quadric surface form a dense  subset of $\wp_4^{-1}(\cD'_{10})$.\ According to Theorem~\ref{th32}, these fourfolds, like those containing a $\sigma$-plane, are all rational; this can also be seen directly by projecting $X$ from the $\P^3$ spanned by $Q$. 
\end{exam}

\begin{exam}\label{ex410}
The GM fourfolds constructed in Example~\ref{ex36} were shown in \cite{hs} to fill out a dense subset of $\wp_4^{-1}(\cD_{20})$.
\end{exam}

\subsection{K3 surfaces associated with GM fourfolds}\label{se46}

For some integers $d$, the nonspecial lattice is isomorphic to the primitive lattice of a polarized K3 surface (with the sign of the intersection form reversed), 
necessarily of degree $d$.\ The precise lattice-theoretic result is the following (\cite[Proposition~6.6]{dim2}).

\begin{prop}\label{p410}
The positive integers $d$ for which the nonspecial lattice is isomorphic to the (opposite of the) primitive cohomology lattice of a degree $d$ polarized K3 surface are the positive integers 
 $d\equiv 2$ or $ 4 \pmod{8}$ such that the only odd  primes that divide $d$ are $\equiv 1\pmod{4}$. 
\end{prop}

The first values of $d$ that satisfy the conditions of the proposition
 and for which the hypersurface $\cD_d$ meets the image of the period map are
$10$, $20$, and $26$.

Let $X$ be a   GM fourfold whose period point belongs to a hypersurface $\cD_d$, for some positive integer $d$ satisfying the conditions of   Proposition~\ref{p410}.\ Since the period map for semipolarized K3 surfaces is bijective, the period point of $X$ in $\cD$ is the period point of a unique K3 surface~$S$ with a nef  (ample when the period of $X$ is general in $\cD_d$) algebraic class $h$ such that $h^2=d$, so that the polarized Hodge structures $K^\bot$  and $H^2(S)_0(-1)$ are isomorphic.\ We say that the K3 surface $S$ is {\em associated} with $X$.\ By Theorem~\ref{th44}, it only depends on the Lagrangian associated with $X$.


\begin{exam}\upshape
In some cases,   the associated K3 surface can be recovered geometrically from the GM fourfold $X$.\ This is the case in Example~\ref{ex35}, and here is another example: when $X$ is general in $\wp_4^{-1}(\cD'_{10})$, it contains a $\tau$-quadric surface $Q$ (see Example~\ref{qua}) and it was shown in  the proof of \cite[Proposition~7.3]{dim2} that the projection $\Bl_QX\to \P^4$ from the $\P^3$ spanned by $Q$ is the blowup of the surface in  $\P^4$ obtained as  the projection of a degree 10 K3 surface $S\subset \P^6$ from the line joining two points of~$S$.\ 

This surface~$S$ is the K3 surface associated with $X$.\ It can be constructed in several other ways:  the dual $A^\bot$ of the Lagrangian $A$ associated with $X$ is in the hypersurface $\Delta$ defined in~\eqref{defd}, so that $Y_{A^\bot}^3\ne\vide$.\ Since $X$ is general in $\wp_4^{-1}(\cD'_{10})$, the set $Y_{A^\bot}^3$ consists of only one point and the associated GM variety has dimension 2 (see Theorem~\ref{th22}): this is the K3 surface~$S$ again; finally, a small resolution of the singular double EPW sextic $\tY_{A^\bot}$ is isomorphic to the Hilbert square $S^{[2]}$ (this is all explained in~\cite{og4}).
\end{exam}

The reason why this strange notion of associated K3 surface is interesting lies in the following conjecture, originally envisioned for cubic fourfolds by Harris--Hassett.

\begin{conj}\label{co413}
A GM fourfold is rational if and only if it has an associated K3 surface.
\end{conj}

This conjecture implies in particular that a very general GM fourfold is not rational.\ We will discuss it at the end of the next section.
 
 \section{Derived categories}\label{se5}
 
 For a very nice geometrically oriented introduction to derived categories, I refer to \cite{huy}.\ For a very detailed series of lecture notes on the topics that concern us here, but in the case of cubic fourfolds, I refer to \cite{kuz}.\ I will only give here a very brief and superficial overview of  the situation for GM varieties.
 
 \subsection{Semiorthogonal decompositions}
 
 Let $\cT$ be a ($\C$-linear) triangulated category.\ We say that a pair $(\cA_1,\cA_2)$ of full triangulated subcategories of $\cT$ forms a (two-step) {\em semiorthogonal decomposition} (s.o.d.\ for short), and we write
 $$\cT=\langle \cA_1,\cA_2\rangle,$$
  if
 \begin{itemize}
\item $\Hom(A_2,A_1) = 0$ for any $A_1\in \cA_1$ and $ A_2\in\cA_2$;
\item $\cA_1$ and $\cA_2$ generate $\cT$, in the sense that for any $T\in\cT$, there is a (unique and functorial) distinguished triangle
$$A_2\to T \to A_1\to A_2[1],$$
with $A_1\in \cA_1$ and $ A_2\in\cA_2$.
\end{itemize}
In that case, one shows (\cite[Lemma~2.2]{kuz}) that $\cA_2$ is a {\em right admissible subcategory} of $\cT$: the inclusion functor $\alpha_2\colon \cA_2\to \cT$ has a right adjoint $\alpha^!_2\colon   \cT\to \cA_2$, given by $T\mapsto A_2$, that satisfies $\alpha^!_2\circ\alpha_2\isom \Id_{\cA_2}$.

Conversely,   any right admissible triangulated functor $\alpha\colon \cA\to \cT$ is fully faithful and there is an s.o.d.
$$\cT=\langle \cA^\perp,\cA\rangle,$$
where $\cA^\perp\coloneqq \Ker(\alpha^!)$ is the {\em right orthogonal} of $\cA$.
 
 This terminology can be generalized to several factors: the equality
 $$\cT=\langle \cA_1,\dots,\cA_r\rangle$$
 means
  \begin{itemize}
\item[(a)] $\Hom(A_j,A_i) = 0$ for any $A_i\in \cA_i$, $ A_j\in\cA_j$, and $i<j$ (no morphisms from right to left);
\item[(b)] $\cA_1,\dots,\cA_r$ generate $\cT$ (this is defined by induction on $r$).
\end{itemize}

A useful example of right admissible subcategory is the following.\ Let $E$ be an object   of~$\cT$ and let $\langle E\rangle$ be the subcategory  of $\cT$ generated by $E$: this is the smallest full triangulated
subcategory of $\cT$ containing $ E$.\ 
If $E$
 is {\em exceptional,} that is, if
$$\Ext^\bullet(E,E)=\C,$$
(this means $\Hom(E,E[m])=\C$ if $m=0$, and $0$ otherwise), the category $\langle E\rangle$ is admissible (\cite[Lemma~1.58]{huy}).\ In particular, we have an s.o.d.
$\cT=\langle \langle E\rangle^\perp,\langle E\rangle\rangle$ which we write simply as
$$\cT=\langle E^\perp,E\rangle.$$

Finally, exceptional objects $E_1,\dots,E_r$  form an {\em exceptional collection} if the subcategories that they generate satisfy   condition (a) above, which means in this case
$\Hom(E_j,E_i[m])=  0$ for all $m$ and all $i<j$.\ There is then an s.o.d.
$$\cT=\langle  \langle E_1,\dots, E_r\rangle^\perp, E_1\dots,E_r\rangle.$$

 \subsection{Examples of semiorthogonal decompositions and exceptional collections}

We will apply these concepts in the following categories: if
  $X$ is a smooth projective (complex) variety,  the  derived category $\Db(X)$ of the category of bounded complexes of coherent sheaves on $X$ is a triangulated category.

 On the projective space $\P^n$, each line bundle $\cO_{\P^n}(m)$ is exceptional and Beilinson proved that there is an s.o.d.\ (\cite[Corollary 8.29]{huy})
$$\Db(X)=\langle \cO_{\P^n},\cO_{\P^n}(1),\dots,\cO_{\P^n}(n)\rangle.$$
Beilinson's result was extended by  Kapranov  to all Grassmannians.

 In the case of $G\coloneqq \Gr(2,V_5)$ which interests us here, Kuznetsov produced the  s.o.d.
$$\Db(G)=\langle \cO_{G},\cU^\vee,\cO_{G}(1),\cU^\vee(1),\dots,\cO_{G}(4),\cU^\vee(4)\rangle$$
(\cite[Section~6.1]{kuh}), where $\cU$ is the rank 2 tautological subbundle on $G$.\ For 
  smooth linear sections  $M\coloneqq \Gr(2,V_5)\cap \P(W_{n+5})$ of $G$ of dimension $n+1\ge 4$, he showed
 $$\Db(M)=\langle \cO_{M},\cU_M^\vee,\cO_{M}(1),\cU_M^\vee(1),\dots,\cO_{M}(n-1),\cU_M^\vee(n-1)\rangle.$$
 Finally, for any GM variety $X$ of dimension $n\ge 3$, he obtained with Perry (\cite[Proposition~2.3]{kp}) that
$$\cO_{X},\cU_X^\vee,\cO_{X}(1),\cU_X^\vee(1),\dots,\cO_{X}(n-3),\cU_X^\vee(n-3)$$
form an exceptional collection in $\Db(X)$.\ If one defines
$$\cA_X\coloneqq \langle \cO_{X},\cU_X^\vee,\cO_{X}(1),\cU_X^\vee(1),\dots,\cO_{X}(n-3),\cU_X^\vee(n-3)\rangle^\bot,$$
the right orthogonal of this exceptional collection, we obtain an s.o.d.
 $$\Db(X)=\langle \cA_X,\cO_{X},\cU_X^\vee,\cO_{X}(1),\cU_X^\vee(1),\dots,\cO_{X}(n-3),\cU_X^\vee(n-3)\rangle.$$
We say that the triangulated category $\cA_X$ is associated with   the GM variety $X$.

 \subsection{Serre functors and K3/Enriques categories}
 
 A {\em Serre functor} for a triangulated category $\cT$ is an equivalence $\Se_\cT\colon \cT\isomto \cT$ with bifunctorial isomorphisms
 $$\Hom(F,\Se_\cT(G))\isomlra \Hom(G,F)^\vee
 $$
 for all  $F,G\in \cT$.\ If a Serre functor exists, it is unique (up to functorial isomorphisms), so it is an invariant of the category.
 
 If $X$ is a smooth projective variety of dimension $n$, the category $\Db(X)$ has a Serre functor given by
 $$\Se_{\Db(X)}(F)=(F\otimes \omega_X)[n],$$
where $\omega_X$ is the dualizing sheaf of $X$.\ In particular, if $\omega_X$ is trivial (one often says that $X$ is a {\em Calabi--Yau variety}), the Serre functor is just the shift by the dimension of $X$.
 
Bondal and Kapranov  proved that an admissible subcategory of a triangulated category that has a Serre functor also has a Serre functor and Kuznetsov computed it in several interesting examples.\ The following is \cite[Proposition~2.6]{kp}.
 
 \begin{prop}
Let $X$ be a GM variety of dimension $n\ge 3$.
\begin{itemize}
\item If $n$ is even, the associated category $\cA_X$ is a {\em K3 category}: it has a Serre functor given by the shift $[2]$.
\item If $n$ is odd, the associated category $\cA_X$ is an {\em Enriques category}: it has a Serre functor $\sigma\circ[2]$, where $\sigma\colon \cA_X\isomto \cA_X$ is a nontrivial involution.
\end{itemize}
\end{prop}

\subsection{The Kuznetsov--Perry conjecture}

Given a GM variety of dimension $n\in\{4,6\}$, we have constructed a K3 category $\cA_X$.\
The question is now the following: can $\cA_X$ be equivalent to the derived category of an actual K3 surface? The answer is no for $X$ very general (\cite[Proposition~2.29]{kp}).
 
 \begin{prop}
If the   category associated with a GM variety $X$ of even dimension is equivalent to the derived category of a K3 surface, the variety $X$ is Hodge special (see Section~\ref{se45}).
\end{prop}

When a GM variety $X$ of dimension 4 has an associated K3 surface $S$ in the Hodge theoretic sense of Section~\ref{se46}, its category $\cA_X$ is equivalent to the category $\Db(S)$.\ However, the converse does not hold.\footnote{To prove these two statements, one has to use \cite[Theorem~1.9]{ppz}, which says that the   category $\cA_X$ associated with a GM variety $X$ of   dimension 4 or 6 is equivalent to the derived category of a K3 surface if and only if a certain lattice $\widetilde H^{1,1}(\cA_X;\Z)$ associated with $\cA_X$ contains a hyperbolic plane.\ 

If $X$ is a GM fourfold with a Hodge theoretically associated K3 surface $S$, the lattice $\widetilde H^{1,1}(\cA_X;\Z)$ does contain a hyperbolic plane by \cite[Theorem~1.2]{per}, and the categories $\cA_X$   and $\Db(S)$ are equivalent.

Then, Pertusi constructed in \cite[Section~3.3]{per}  GM fourfolds $X$ such that $\widetilde H^{1,1}(\cA_X;\Z)$ does contain a hyperbolic plane (so that the category $\cA_X$   is equivalent to the derived category of a K3 surface), but with no Hodge theoretically associated K3 surface.\ Many thanks to Alex Perry for these explanations.}
 
Therefore (\cite[Conjecture~3.12]{kp}, \cite[Remark~1.10]{ppz}), there are  two competing conjectural conditions
(existence of a categorical versus Hodge-theoretic K3 surface) for the rationality of a GM
fourfold.\ This should be contrasted with the case of cubic fourfolds, where these conditions
are known to be equivalent.

 \end{document}